\documentclass[11pt,reqno,noamsfonts]{amsart}

\usepackage[T1]{fontenc}
\usepackage[utf8]{inputenc}

\usepackage{newtxtext}
\usepackage{newtxmath}

\usepackage{microtype,hyperref}

\usepackage[
  paperwidth=6.15in,
  paperheight=9.25in,
  inner=0.52in,
  outer=0.52in,
  top=0.78in,
  bottom=0.82in,
  footskip=0.48in
]{geometry}

\usepackage{amsmath,amsthm,mathtools,mathrsfs,bm}

\usepackage{booktabs}
\usepackage{enumitem}
\usepackage{xcolor}

\usepackage{listings}
\usepackage[scaled=0.88]{inconsolata}

\definecolor{leankeyword}{RGB}{36,73,143}
\definecolor{leancomment}{RGB}{80,120,80}
\definecolor{leanstring}{RGB}{145,65,35}
\definecolor{leanrule}{RGB}{190,190,190}
\definecolor{leanback}{RGB}{248,248,248}

\lstdefinelanguage{Lean4}{
  sensitive=true,
  morekeywords={
    import,set_option,namespace,end,abbrev,def,theorem,lemma,
    structure,class,instance,where,extends,variable,section,
    by,have,show,exact,refine,intro,rintro,obtain,rcases,
    let,fun,if,then,else,match,with,do,return,from,
    Type,Prop,true,false
  },
  morecomment=[l]{--},
  morecomment=[s]{/-}{-/},
  morestring=[b]"
}

\lstdefinestyle{lean4style}{
  language=Lean4,
  basicstyle=\ttfamily\scriptsize,
  keywordstyle=\color{leankeyword}\bfseries,
  commentstyle=\color{leancomment}\itshape,
  stringstyle=\color{leanstring},
  backgroundcolor=\color{leanback},
  frame=single,
  rulecolor=\color{leanrule},
  numbers=left,
  numberstyle=\tiny\color{black!55},
  numbersep=8pt,
  xleftmargin=1.8em,
  framexleftmargin=1.4em,
  breaklines=true,
  breakatwhitespace=false,
  columns=fullflexible,
  keepspaces=true,
  showstringspaces=false,
  tabsize=2,
  upquote=true,
  literate=
    {¬}{{$\neg$}}1
    {ε}{{$\varepsilon$}}1
    {η}{{$\eta$}}1
    {⁻¹}{{$^{-1}$}}2
    {₁}{{$_1$}}1
    {₂}{{$_2$}}1
    {ℝ}{{$\mathbb{R}$}}1
    {→}{{$\to$}}1
    {∀}{{$\forall$}}1
    {∃}{{$\exists$}}1
    {∞}{{$\infty$}}1
    {∧}{{$\wedge$}}1
    {≤}{{$\le$}}1
    {⟨}{{$\langle$}}1
    {⟩}{{$\rangle$}}1
    {ν}{{$\nu$}}1
    {₀}{{$_0$}}1
    {ℕ}{{$\mathbb{N}$}}1
    {∈}{{$\in$}}1
    {⊆}{{$\subseteq$}}1
}

\usepackage{tikz}
\usetikzlibrary{arrows.meta,calc,patterns,positioning}

\usepackage{titlesec}

\titleformat{\section}
  {\centering\normalfont\large\bfseries}
  {\thesection.}
  {0.55em}
  {}

\titlespacing*{\section}
  {0pt}
  {2.2em}
  {1.25em}

\titleformat{\subsection}
  {\normalfont\normalsize\bfseries}
  {\thesubsection.}
  {0.55em}
  {}

\titlespacing*{\subsection}
  {0pt}
  {1.7em}
  {0.7em}

\usepackage[nameinlink,capitalize,noabbrev]{cleveref}

\hypersetup{
  colorlinks=true,
  linkcolor=blue!45!black,
  citecolor=green!30!black,
  urlcolor=blue!55!black,
  pdftitle={The Physical Cutoff Does Not Restore Homogenization},
  pdfauthor={Michele Caprio}
}

\allowdisplaybreaks

\setlist[itemize]{
  leftmargin=2em,
  itemsep=0.25em,
  topsep=0.4em
}

\setlist[enumerate]{
  leftmargin=2.35em,
  itemsep=0.25em,
  topsep=0.4em
}

\newtheorem{theorem}{Theorem}[section]
\newtheorem{proposition}[theorem]{Proposition}
\newtheorem{lemma}[theorem]{Lemma}
\newtheorem{corollary}[theorem]{Corollary}

\theoremstyle{definition}
\newtheorem{definition}[theorem]{Definition}

\theoremstyle{remark}
\newtheorem{remark}[theorem]{Remark}

\AddToHook{env/proposition/begin}{\crefalias{theorem}{proposition}}
\AddToHook{env/lemma/begin}{\crefalias{theorem}{lemma}}
\AddToHook{env/corollary/begin}{\crefalias{theorem}{corollary}}
\AddToHook{env/conjecture/begin}{\crefalias{theorem}{conjecture}}
\AddToHook{env/assumption/begin}{\crefalias{theorem}{assumption}}
\AddToHook{env/definition/begin}{\crefalias{theorem}{definition}}
\AddToHook{env/example/begin}{\crefalias{theorem}{example}}
\AddToHook{env/remark/begin}{\crefalias{theorem}{remark}}

\newcommand{\R}{\mathbb R}
\newcommand{\T}{\mathbb T}
\newcommand{\Pp}{\mathcal P}

\newcommand{\dist}{\operatorname{dist}}
\newcommand{\osc}{\operatorname{osc}}

\newcommand{\clco}{\overline{\operatorname{co}}}
\newcommand{\Dplus}{D^+}
\newcommand{\e}{\varepsilon}

\newcommand{\Hunc}{H_{\mathrm{unc}}}
\newcommand{\Hcut}{H_{+}}
\newcommand{\Hhat}{\widehat H}
\newcommand{\Gunc}{G_{\mathrm{unc}}}
\newcommand{\Gcut}{G_{+}}

\newcommand{\upperE}{\overline{\mathbb E}}

\title[The Physical Cutoff Does Not Restore Homogenization]
{The Physical Cutoff Does Not Restore Homogenization:\\
Phase-Dependent Burning in the Strain $G$-Equation}

\author{Michele Caprio}
\address{Department of Computer Science, University of Warwick, United Kingdom}
\email{Michele.Caprio@warwick.ac.uk}

\date{}

\begin{document}

\begin{abstract}
We disprove the expectation stated by Xin, Yu, and Ronney that the
physical positive part strain $G$-equation should possess an effective
burning velocity in cellular flows.  For the standard cellular flow in
dimension two
  $V_A(x_1,x_2)
  =A(-\sin x_1\cos x_2,\cos x_1\sin x_2)$,
if
\[
  0<d<\frac{20}{399},
  \qquad
  \sqrt{1+4d^2}<Ad\le1+\frac d{10},
\]
then, for every unit planar slope, the periodic correction develops
oscillations at least linearly in time.  The solution remains bounded
below on an explicit horizontal channel through $(\pi,0)$, while at
$(\pi/2,0)$ it decreases at rate at least $CA/\log A$, with $C>0$
universal.  The same conclusions hold for arbitrary continuous periodic
perturbations of planar initial data.  Under the physical scaling
$V_A(x/\varepsilon)$ and $d_\varepsilon=\varepsilon d$, an order one
value gap persists between points at distance $O(\varepsilon)$ at every
positive macroscopic time, so the rescaled solutions have no locally
uniformly convergent subsequence.
The proof uses the Hamiltonian sandwich
$H_{\mathrm{unc}}\le H_+\le\widehat H$.  The upper comparator $\widehat H$ is a
rectangular support function, equivalently an upper expectation over a
state-dependent credal set, whose reversed control dynamics possess an
invariant comparison channel.  We also prove that, for any $C^2$
incompressible periodic flow, every $\varepsilon$-outward barrier
certificate has covering radius at most $2d\varepsilon$ for all
sufficiently small $\varepsilon$.  We further discuss implications for
statistics and machine learning: the construction shows that
rectangular, time-consistent local uncertainty need not imply forgetting
of the initial state in the long run, highlighting the need for
additional global stability or ergodicity conditions in robust
sequential decision making.  Two Lean~4 appendices record conditional
formalizations of a sufficient $p=e_1$ subregime of the nonhomogenization result
and of the logical assembly of the rigidity theorem for barrier
certificates.
\end{abstract}

\maketitle

\begingroup
\setcounter{tocdepth}{1}
\renewcommand{\numberline}[1]{}
\tableofcontents
\endgroup

\section{Introduction}\label{sec:introduction}
A premixed flame is a flame in which the fuel and oxidizer are already
mixed before combustion occurs. Such a flame can be idealized as a moving surface separating burned
from unburned material (more precisely, a hypersurface in an 
$n$-dimensional space). Let
$X\subseteq\mathbb R^n$, $n\geq2$, denote the spatial domain. In a
level set description, the evolving flame is encoded by a scalar
function
\[
  G:X\times\mathbb R_+\to\mathbb R.
\]
Thus, to every position $x$ and time $t$, the function assigns a real
number $G(x,t)$. With a fixed sign convention, $G(x,t)<0$ means that
$x$ lies in the burned region at time $t$, $G(x,t)>0$ means that it lies
in the unburned region, and
\[
  \Gamma_t=\{x:G(x,t)=0\}
\]
is the flame front itself.
The motion of the surrounding gas is described by a velocity field
\[
  V:X\to\mathbb R^n.
\]
For every spatial point $x$, the vector $V(x)$ specifies the speed and
direction in which the gas near $x$ is moving; its $n$ components give
the velocities along the $n$ spatial directions. If the gas were moving
everywhere with exactly the same velocity, it would simply carry the
flame along. In general, however, the velocity varies from one point to
another, and this variation can also stretch, compress, or rotate small
pieces of gas.

Before accounting for how stretching and compression affect the burning
speed, the simplest first order model for the flame is the inviscid
$G$-equation
\begin{align}\label{inviscid-eq}
    G_t+|DG|+V(x)\cdot DG=0.
\end{align}
Here $G_t=\partial G/\partial t$ describes how the level set function changes with time, while
$DG = (\partial G/\partial x_1, \ldots, \partial G/\partial x_n )$ is its spatial gradient. At a point of the flame front for which
$DG\neq0$,
  $\nu=\frac{DG}{|DG|}$
is a unit vector perpendicular to the front.

The flame moves for two distinct reasons. First, even if the surrounding
gas were completely at rest, combustion would make the flame spread into
the unburned mixture. The speed at which a smooth flame propagates through
a motionless premixed gas is called its {\em laminar burning speed}. We use
this speed as our unit of velocity, so that it is represented by a factor $1$ in the equation. This burning-induced motion is represented
by the term $|DG|$. Second, the surrounding gas may itself be moving and
therefore carries the flame front with it. This transport by the gas is
represented by the term $V(x)\cdot DG$. Thus the flame both burns through
the fresh mixture and is carried by the surrounding flow. The term
``inviscid'' refers here to this first order $G$-equation, before
diffusive or curvature corrections are added.

The basic equation \eqref{inviscid-eq} does not yet account for the fact that the burning
speed itself can be affected by the way the surrounding gas is moving.
Nearby points of the gas may have slightly different velocities, so that
a small piece of material can be stretched in some directions, compressed
in others, or rotated. To describe these local changes in velocity, we
use the Jacobian matrix
\[
  DV(x)=
  \left(
    \frac{\partial V_i}{\partial x_j}(x)
  \right)_{i,j}.
\]
Its entries record how each component of the gas velocity changes as one
moves in each spatial direction.

The matrix $DV(x)$ contains both deformation and rotation. It can be
decomposed as
\[
  DV(x)
  =
  \frac{DV(x)+DV(x)^{\mathsf T}}2
  +
  \frac{DV(x)-DV(x)^{\mathsf T}}2.
\]
The second matrix (the one after the $+$ sign) describes local rigid rotation: it can change the
direction joining two nearby fluid particles without changing their
distance. The first matrix,
\[
  S_V(x)
  \coloneqq 
  \frac{DV(x)+DV(x)^{\mathsf T}}2,
\]
is called the {\em strain tensor}. It describes the part of the local fluid
motion that changes distances between nearby particles and hence
stretches or compresses material.
At the flame front, the scalar
\[
  \nu^{\mathsf T}S_V(x)\nu
\]
measures this stretching or compression in the direction $\nu$
perpendicular to the front. This deformation can change the rate at
which the flame burns through the unburned mixture.

A standard first order model assumes that the resulting change in
burning speed is proportional to this stretching or compression in the
direction perpendicular to the flame front. The positive parameter $d$, called the {\em Markstein length} \cite{Markstein1964}, determines how
strongly this deformation affects the burning speed. Starting from the baseline burning speed $1$, the
strain-corrected speed is therefore
\[
  1+d \nu^{\mathsf T}S_V(x)\nu.
\]
Since
  $\nu=\frac{DG}{|DG|}$,
this can equivalently be written as
\[
  1+d\frac{DG^{\mathsf T}S_V(x)DG}{|DG|^2}.
\]

Strong compression can make this strain-corrected burning speed
negative. Taken literally, that would assign the flame a negative
burning speed, corresponding to a reversal of burning rather than local
extinction.
The physical model therefore truncates the corrected burning speed at
zero. Writing
  $(a)_+\coloneqq \max\{a,0\}$,
one obtains the physical strain $G$-equation
\begin{equation}\label{eq:physical-intro}
  G_t+
  \left(
    1+d\frac{DG^{\mathsf T}S_V(x)DG}{|DG|^2}
  \right)_+|DG|
  +V(x)\cdot DG=0.
\end{equation}
We can see this as a refinement of the inviscid equation \eqref{inviscid-eq}. Thus, sufficiently strong compression may completely suppress the
flame's  local burning motion, but it cannot make the flame burn
backwards. Even where burning is locally suppressed, the surrounding
gas may still move the front through the transport term $V(x)\cdot DG$. The geometry and the different contributions to the motion of the
flame front are summarized in \Cref{fig:strain-G-geometry}.

\begin{figure}[h!]
\centering
\includegraphics[width=.8\textwidth]{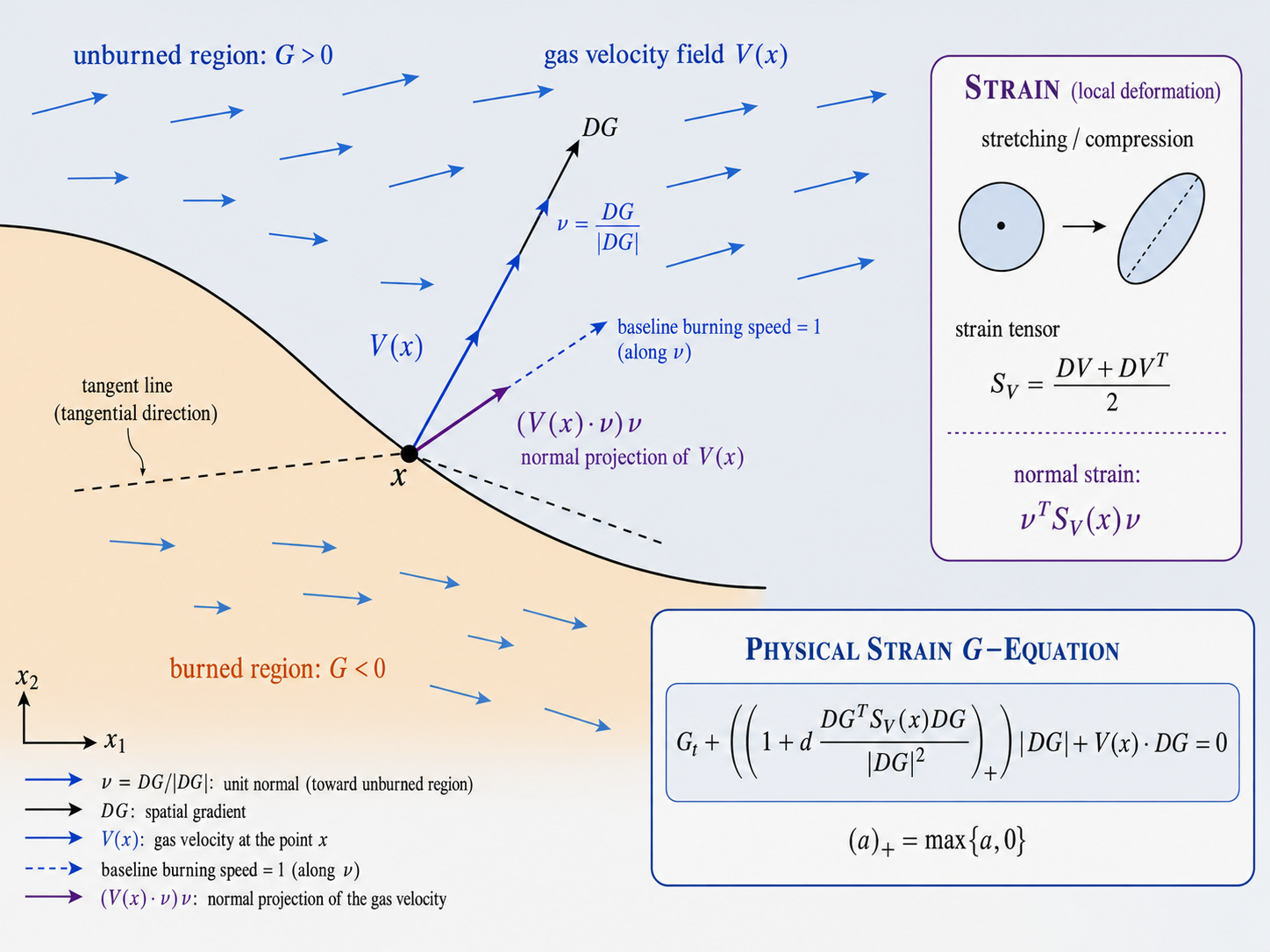}
\caption{
Schematic of the physical strain $G$-equation \eqref{eq:physical-intro}. The zero level set
$\Gamma_t$ represents the flame front, with unit normal
$\nu=DG/|DG|$. The surrounding flow $V$ carries the front, while the
normal strain $\nu^{\mathsf T}S_V\nu$ modifies its local burning speed.
The positive part cutoff suppresses burning under sufficiently strong
compression without allowing reversal.
}
\label{fig:strain-G-geometry}
\end{figure}

It is useful to write \eqref{eq:physical-intro} in the Hamilton-Jacobi
form
\[
  G_t+H(x,DG)=0,
\]
where
\[
  H(x,q)
  =
  \left(
    1+d\frac{q^{\mathsf T}S_V(x)q}{|q|^2}
  \right)_+|q|
  +V(x)\cdot q,
  \qquad q\neq0,
\]
with $H(x,0)=0$. The function $H$ is called the {\em Hamiltonian}.

Two features of this Hamiltonian make it particularly difficult to determine whether the flame has a single effective propagation speed at large times. First, in some directions the surrounding gas may compress the flame
strongly enough that
\[
  1+d \nu^{\mathsf T}S_V(x)\nu\le0.
\]
The positive part operation in \eqref{eq:physical-intro} then replaces
this nonpositive value by zero. Thus the local burning speed vanishes
in that direction. The same is true in the opposite direction $-\nu$,
because
  $(-\nu)^{\mathsf T}S_V(x)(-\nu)
  =
  \nu^{\mathsf T}S_V(x)\nu$.
Once burning has vanished, only the contribution from the motion of the
surrounding gas remains. This contribution has opposite signs in the
two directions $\nu$ and $-\nu$, since
  $V(x)\cdot(-\nu)=-V(x)\cdot\nu$.
Therefore, in at least one of these two directions, it is nonpositive.
Increasing the magnitude $|q|$ of the gradient in that direction does
not force $H(x,q)$ to become arbitrarily large and positive. In
Hamilton-Jacobi terminology, this possible failure of $H(x,q)$ to tend
to $+\infty$ as $|q|\to\infty$ is called \emph{noncoercivity}.

Second, the strain correction depends on the direction $q/|q|$ of the
gradient. Changing this direction changes how strongly the surrounding
flow stretches or compresses the flame. In some directions the flame
burns normally, while in others compression can reduce the burning speed
all the way to zero. Consequently, the Hamiltonian can change sharply
as the gradient direction varies, and its value at an intermediate
direction need not lie below the corresponding average of its values in
two other directions. In Hamilton-Jacobi terminology, $H(x,q)$ is
therefore generally {\em nonconvex in $q$}. The simultaneous absence of coercivity and convexity therefore removes
two of the main mathematical structures normally used to understand
the large scale behavior of such equations.

To study the large scale motion of the flame, we start from the simplest
possible initial shape: a flat front. Fix a unit vector
$p\in\mathbb R^n$ and take
  $G(x,0)=p\cdot x$.
Since the flame front is the zero level set of $G$, initially
  $\Gamma_0=\{x:p\cdot x=0\}$.
This is a hyperplane perpendicular to $p$;
thus, the initial flame has
no wrinkles or curvature, and $p$ specifies the direction perpendicular
to it. Initial data of this form are called \emph{planar initial data}.

We also assume that the surrounding flow is periodic, meaning that the
same pattern of gas motion repeats from one spatial cell to the next.
After choosing the size of a periodic cell as the unit of length, this
can be written as
\[
  V(x+k)=V(x),
  \qquad\text{for every }k\in\mathbb Z^n.
\]
Thus the flame repeatedly encounters the same pattern of transport,
stretching, and compression as it moves through space.

Although $G(x,t)$ itself is not periodic because of the linear term
$p\cdot x$, its deviation from the original plane,
\[
  u(x,t)\coloneqq G(x,t)-p\cdot x,
\]
is periodic. In other words, the flow may wrinkle the initially flat
flame inside each periodic cell, but the same cell scale pattern of
deviations is repeated throughout space.

The central large scale question we address in this work is whether the cell scale wrinkling of
the flame eventually becomes negligible when compared with its overall
motion. More precisely, we ask whether there exists a single number
$\overline H(p)$ such that
\begin{equation}\label{eq:effective-speed-intro}
  \frac{G(x,t)-p\cdot x}{t}
  \longrightarrow -\overline H(p)
  \qquad\text{uniformly in }x
  \text{ as }t\to\infty.
\end{equation}
If this happens, then in the long run
\[
  G(x,t)\approx p\cdot x-t \overline H(p).
\]
The corresponding zero level set, then, becomes approximately
\[
  \{x:p\cdot x=t \overline H(p)\},
\]
which is a flat front moving in the direction $p$ at the constant speed
$\overline H(p)$. Thus, despite all the complicated wrinkling produced
at the scale of individual periodic cells, the flame has a single
deterministic propagation speed when viewed over sufficiently large
distances and times. The number $\overline H(p)$ is called the
\emph{effective burning velocity} in the direction $p$.

The word ``uniformly'' in \eqref{eq:effective-speed-intro} is important.
It means that the same long run rate is obtained regardless of where
$x$ lies inside the periodic cell; equivalently,
\[
  \sup_x
  \left|
    \frac{G(x,t)-p\cdot x}{t}
    +\overline H(p)
  \right|
  \longrightarrow0.
\]
Thus, no particular location within the repeated microscopic environment
is allowed to retain a different asymptotic propagation rate.

The flows considered in the open problem are also assumed to be
\emph{incompressible}. This means that small volumes of gas are
preserved as they move with the flow, even though they may be stretched
in some directions and compressed in others. Mathematically,
incompressibility is expressed by
\[
  \operatorname{div}V
  \coloneqq 
  \sum_{i=1}^n
  \frac{\partial V_i}{\partial x_i}
  =0.
\]

Xin, Yu, and Ronney explicitly expected the effective burning velocity
to exist for cellular flows and asked in Question~11 whether it exists
for general incompressible periodic flows in all dimensions
\cite[Section~5.2, Question~11]{XinYuRonney2024}.  We disprove that
expectation already for the standard two dimensional cellular flow.
Writing $\mu=Ad$, if
\[
  0<d<\frac{20}{399},
  \qquad
  \sqrt{1+4d^2}<\mu\le1+\frac d{10},
\]
then, for every unit planar direction $p$, the periodic correction
$u^p(x,t)=G_+^p(x,t)-p\cdot x$ satisfies
\[
  \liminf_{t\to\infty}
  \frac{\osc_{\T^2_{2\pi}}u^p(\cdot,t)}{t}
  \ge C\frac A{\log A},
\]
where $\osc_{\T^2_{2\pi}} u^p(\cdot,t)
\coloneqq 
\sup_{x\in\T^2_{2\pi}} u^p(x,t) - \inf_{x\in\T^2_{2\pi}} u^p(x,t)$ measures the gap between the largest and smallest values of the periodic correction across one $2\pi\times 2\pi$ periodic cell at time $t$, and $C>0$ is universal.  The slow estimate holds on an explicit whole
channel through $(\pi,0)$, not merely at one point.  The same phase
separation persists under arbitrary continuous periodic perturbations
of the planar initial datum.  After the physical small scale rescaling,
an order one value gap persists across an $O(\varepsilon)$ distance at
every positive macroscopic time, which rules out locally uniform
subsequential convergence.  We also prove a general result showing that, in incompressible periodic
flows, any region that nearly blocks outward propagation must come close
to every point within one period of the flow. The cellular
flow and the precise statements are given in \Cref{main-overv}.

\subsection{Organization}

\Cref{main-overv} states the two main results and explains the proof
strategy, while \Cref{sec:literature} reviews the related literature.
Sections~\ref{sec:setting}-\ref{completion} prove linear microscopic
phase separation in every planar direction, and its stability under
periodic perturbations of planar data.  \Cref{sec:no barrier} proves the
rigidity theorem for outward barrier certificates, and
\Cref{sec:consequences} develops the small scale noncompactness and
fixed pair consequences.  Appendix~\ref{app:background} collects the needed 
background and conventions, Appendix~\ref{app:control} proves the
control representation for planar periodic data, and
Appendices~\ref{app:lean-formalization} and
\ref{app:lean-no-barrier-formalization} record conditional Lean~4
formalizations of a sufficient $p=e_1$ subregime of the nonhomogenization result
and of the logical assembly of the barrier certificate rigidity theorem,
respectively.

\section{Main Theorems: Overview}\label{main-overv}
We specialize to two spatial dimensions and construct the flow that
will provide our counterexample.  We use the standard two dimensional cellular flow
\begin{equation}\label{cellular_flow}
    V_0(x_1,x_2)
  \coloneqq 
  \bigl(
    -\sin x_1\cos x_2, 
    \cos x_1\sin x_2
  \bigr).
\end{equation}
This velocity field generates a repeating array of circulating cells; it is the cellular flow studied by Xin and Yu
\cite{XinYu2014}. Because sine and cosine are $2\pi$-periodic, this flow repeats its
spatial pattern every $2\pi$ in each coordinate. We can therefore
identify points that differ by whole periods and work on the
$2\pi$-periodic torus
\[
  \mathbb T^2_{2\pi}
  \coloneqq 
  (\mathbb R/2\pi\mathbb Z)^2,
\]
which can be viewed as a single $2\pi\times2\pi$ periodic cell with
opposite edges identified.
For $A>0$, we set
\[
  V_A\coloneqq A V_0,
\]
so that $A$ controls the strength of the flow without changing its
spatial pattern.
Since its components are products of sine and cosine functions,
$V_0$ is smooth (indeed, $C^\infty$), and therefore so is
$V_A=A V_0$. It is also
incompressible, since
\[
\begin{aligned}
  \operatorname{div}V_A
  &=
  A\left[
    \frac{\partial}{\partial x_1}
      (-\sin x_1\cos x_2)
    +
    \frac{\partial}{\partial x_2}
      (\cos x_1\sin x_2)
  \right]\\
  &=A\bigl(
    -\cos x_1\cos x_2
    +\cos x_1\cos x_2
  \bigr)
  =0.
\end{aligned}
\]
Thus $V_A$ satisfies precisely the smoothness, periodicity, and incompressibility
assumptions imposed on the flows considered in \cite[Question~11]{XinYuRonney2024}.
Its Jacobian matrix is
\[
  DV_A(x)
  =
  A
  \begin{pmatrix}
    -\cos x_1\cos x_2 & \sin x_1\sin x_2\\
    -\sin x_1\sin x_2 & \cos x_1\cos x_2
  \end{pmatrix}.
\]
Taking its symmetric part gives the strain tensor
\begin{equation}\label{eq:strain-tensor-intro}
  S_{V_A}(x)
  =
  A
  \begin{pmatrix}
    -\Phi(x)&0\\
    0&\Phi(x)
  \end{pmatrix},
  \qquad
  \Phi(x)\coloneqq \cos x_1\cos x_2.
\end{equation}
In particular, the strain tensor has no off-diagonal terms in these
coordinates, a simple structure that will be central to the trapping
mechanism in our counterexample.

\begin{theorem}[Linear Microscopic Phase Separation in Every Direction]
\label{thm:main}
Let
\begin{equation}\label{eq:parameter-window}
  0<d<\frac{20}{399},
  \qquad
  \sqrt{1+4d^2}<\mu\coloneqq Ad\le1+\frac d{10}.
\end{equation}
Set
\[
  r_\mu\coloneqq\arccos\frac1\mu,
  \qquad
  \mathcal D_\mu
  \coloneqq[\pi-r_\mu,\pi+r_\mu]\times\{0\},
\]
and
\[
  x_{\mathrm s}\coloneqq(\pi,0),
  \qquad
  x_{\mathrm f}\coloneqq\left(\frac\pi2,0\right).
\]
For every unit vector $p=(p_1,p_2)\in\mathbb R^2$, let $G_+^p$ denote
the viscosity solution of
\begin{equation}\label{eq:cutoff-main}
\begin{cases}
(G_+^p)_t+
\left(
  1+d\dfrac{D(G_+^p)^{\mathsf T}
  S_{V_A}(x)D G_+^p}{|D G_+^p|^2}
\right)_+
|D G_+^p|
+V_A(x)\cdot D G_+^p=0,\\[0.4em]
G_+^p(x,0)=p\cdot x.
\end{cases}
\end{equation}
At points where $D G_+^p=0$, the Hamiltonian is understood through its
continuous extension, whose value at zero gradient is $0$.  Then,
\begin{equation}\label{eq:slow-main}
  G_+^p(x,t)
  \ge \pi p_1-r_\mu|p_1|,
  \qquad
  x\in\mathcal D_\mu,
  \quad t\ge0,
\end{equation}
and
\begin{equation}\label{eq:fast-main}
  \liminf_{t\to\infty}
  -\frac{G_+^p(x_{\mathrm f},t)}t
  \ge C\frac A{\log A},
\end{equation}
where $C>0$ is universal and independent of $p$, $A$, and $d$.
Consequently, for
$u^p(x,t)=G_+^p(x,t)-p\cdot x$,
\begin{equation}\label{eq:oscillation-main}
  \liminf_{t\to\infty}
  \frac{\osc_{\T^2_{2\pi}}u^p(\cdot,t)}t
  \ge C\frac A{\log A}.
\end{equation}
In particular, no phase-independent effective burning velocity exists.
\end{theorem}

The interval in \eqref{eq:parameter-window} is nonempty precisely under
the stated restriction $d<20/399$, because
\[
  \sqrt{1+4d^2}<1+\frac d{10}
  \quad\Longleftrightarrow\quad
  d<\frac{20}{399}.
\]
It lies above the threshold $Ad=1$ at which the uncut strain law first
permits negative local burning velocities. Two corollaries are then of particular interest.

\begin{corollary}[Periodic Perturbations of Planar Data]
\label{cor:periodic-perturbations}
Under the hypotheses of \Cref{thm:main}, let
$u_0\in C(\T^2_{2\pi})$, and let $G_{+,u_0}^p$ solve the physical
cutoff equation with initial datum
\[
  G_{+,u_0}^p(x,0)=p\cdot x+u_0(x).
\]
If $m=\min u_0$ and $M=\max u_0$, then
\begin{equation}\label{eq:periodic-data-sandwich}
  G_+^p(x,t)+m
  \le G_{+,u_0}^p(x,t)
  \le G_+^p(x,t)+M.
\end{equation}
Hence \eqref{eq:fast-main} and \eqref{eq:oscillation-main} remain valid
with $G_+^p$ replaced by $G_{+,u_0}^p$, while the right hand side of
\eqref{eq:slow-main} is replaced by
$\pi p_1-r_\mu|p_1|+m$.
\end{corollary}

\begin{corollary}[Order One Separation at Vanishing Spatial Scale]
\label{cor:homogenization-fails}
Under the hypotheses of \Cref{thm:main}, fix a unit vector
$p\in\mathbb R^2$.  For $\varepsilon>0$, set
\[
  V_A^\varepsilon(x)\coloneqq V_A(x/\varepsilon),
  \qquad
  d_\varepsilon\coloneqq \varepsilon d,
\]
and let $G^{\varepsilon,p}$ solve
\begin{equation}\label{eq:scaled-physical}
\begin{cases}
(G^{\varepsilon,p})_t+
\left(
1+d_\varepsilon
\dfrac{
D(G^{\varepsilon,p})^{\mathsf T}
S_{V_A^\varepsilon}(x)
D G^{\varepsilon,p}
}{
|D G^{\varepsilon,p}|^2
}
\right)_+
|D G^{\varepsilon,p}|\\
\hfill{}+
V_A^\varepsilon(x)\cdot D G^{\varepsilon,p}=0,\\[0.5em]
G^{\varepsilon,p}(x,0)=p\cdot x.
\end{cases}
\end{equation}
Then, for every $t_0>0$,
\begin{equation}\label{eq:scaled-gap-main}
  \liminf_{\varepsilon\downarrow0}
  \left[
  G^{\varepsilon,p}(\varepsilon x_{\mathrm s},t_0)
  -G^{\varepsilon,p}(\varepsilon x_{\mathrm f},t_0)
  \right]
  \ge C\frac A{\log A}\,t_0,
\end{equation}
while
\[
  |\varepsilon x_{\mathrm s}-\varepsilon x_{\mathrm f}|
  =\frac\pi2\varepsilon.
\]
Consequently, no sequence $\varepsilon_j\downarrow0$ yields locally
uniform convergence of $G^{\varepsilon_j,p}$ on
$\mathbb R^2\times[0,\infty)$.  In particular, the physical strain
equation does not homogenize to a deterministic first order
Hamilton-Jacobi equation for any unit planar slope $p$.
\end{corollary}

Our second main result is independent of the cellular construction.  It
shows that incompressibility and the physical cutoff impose a strong
global rigidity on approximate outward Hamiltonian certificates.

\begin{definition}[Outward Barrier Certificate]
\label{def:barrier certificate}
Let $H:\T_L^n\times\mathbb R^n\to\mathbb R$ be positively homogeneous
in its second argument.  A nonempty closed set $D\subset\T_L^n$ is an
$\varepsilon$-outward barrier certificate for $H$ if
\[
  H(x,\nu)\le\varepsilon
\]
for every $x\in D$ and every unit proximal outward normal $\nu$ of $D$
at $x$.
\end{definition}

For a smooth front represented as $\{G=0\}$ and oriented by
$\nu=DG/|DG|$, the equation and positive homogeneity assign the normal
velocity $H(x,\nu)$.  Hence, if $D$ is the closure of a $C^1$ domain
represented locally as $\{G\le0\}$, then $\nu$ is the outward unit
normal and $H(x,\nu)\le0$ is the pointwise first order condition that
the prescribed outward normal velocity be nonpositive along
$\partial D$.  The definition records only this normal Hamiltonian
inequality; it neither asserts invariance of $D$ under the fluid flow,
nor characterizes every possible global blocking mechanism.

\begin{theorem}[Rigidity of Outward Barrier Certificates]
\label{thm:no barrier}
Let $n\ge2$ and $L>0$, and write
$\T_L^n\coloneqq\mathbb R^n/L\mathbb Z^n$.  Let
$V\in C^2(\T_L^n;\mathbb R^n)$ be divergence free, let $d>0$, and
define
\[
  H(x,q)
  \coloneqq
  \left(
    1+d\frac{q^{\mathsf T}S_V(x)q}{|q|^2}
  \right)_+|q|
  +V(x)\cdot q,
  \qquad H(x,0)\coloneqq0.
\]
We denote by $\dist$ the geodesic distance on the flat torus
$\T_L^n$. For a nonempty closed set $D\subset\T_L^n$, set
\[
  \dist(y,D)\coloneqq\inf_{x\in D}\dist(y,x),
  \qquad
  R(D)\coloneqq\max_{y\in\T_L^n}\dist(y,D).
\]
There exists $\varepsilon_0>0$, depending only on
$d$, $L$, and $\|V\|_{C^2}$, such that every
$\varepsilon$-outward barrier certificate with
$0\le\varepsilon\le\varepsilon_0$ satisfies
\[
  R(D)\le2d\varepsilon.
\]
In particular, every exact outward barrier certificate is the whole
torus.
\end{theorem}

The terminology concerning proximal normals is recalled in
Appendix~\ref{app:proximal-background}; the proof is given in
\Cref{sec:no barrier}.

\begin{corollary}[Asymptotic Density of Approximate Certificates]
\label{cor:barriers-dense}
Under the assumptions and notation of \Cref{thm:no barrier}, let
$\varepsilon_j\downarrow0$ and let $D_j\subset\T_L^n$ be
$\varepsilon_j$-outward barrier certificates.  Then,
\[
  R(D_j)\longrightarrow0.
\]
Thus, no sequence of increasingly accurate outward certificates can
retain a complement of uniformly positive macroscopic size.
\end{corollary}

Let us summarize the two main results intuitively. We start from an
initially flat flame front, which the surrounding flow wrinkles as it
evolves. The question is whether, in the long run, those wrinkles become
negligible at the macroscopic scale, so that the whole flame effectively
advances at one constant rate. Our counterexample says no: the flame can
exhibit different long run propagation rates at different spatial
locations within the periodic flow, so there is no single constant
effective speed describing the front.

The slow phase is not caused by a localized region that blocks the
flame from spreading outward.  Our \Cref{thm:no barrier} shows that any region
that blocks outward propagation in the precise sense considered here
would have to occupy the whole periodic domain.  The slow phase instead
comes from the way the flame and the surrounding flow interact.

\subsection{The Proof of \Cref{thm:main} in Four Moves}\label{proof-overview}
The argument is short once the correct comparisons are placed in the correct order. Fix a unit vector $p\in\mathbb R^2$ throughout this overview.

\textit{1. Transfer the Fast Propagation Estimate.}
We first compare the physical strain equation \eqref{eq:cutoff-main} with the equation studied
by Xin and Yu \cite{XinYu2014}, in which the positive part operation is
removed. Recall that
  $(a)_+\coloneqq \max\{a,0\}$,
so that $(a)_+\ge a$ for every real number $a$. Consequently, if $H$ denotes the
Hamiltonian of the physical strain equation \eqref{eq:cutoff-main} and $H_{\mathrm{unc}}$ the
corresponding Hamiltonian without the positive part, then
\[
  H_{\mathrm{unc}}(x,q)\le H(x,q),
  \qquad\text{for every }x,q.
\]
For Hamilton-Jacobi equations with the same initial data, this ordering
of the Hamiltonians reverses at the level of the solutions. Indeed, an
equation of the form
  $u_t+H(x,Du)=0$
can be read formally as $u_t=-H(x,Du)$; a larger Hamiltonian therefore
pushes the solution downward faster. More precisely, if
$H_1(x,q)\le H_2(x,q)$ and $u_2$ solves the equation with Hamiltonian
$H_2$, then
\[
  (u_2)_t+H_1(x,Du_2)
  \le
  (u_2)_t+H_2(x,Du_2)
  =0.
\]
Thus, $u_2$ is a subsolution of the equation with Hamiltonian $H_1$.
The comparison principle then gives
  $u_2(x,t)\le u_1(x,t)$,
where $u_1$ is the solution with Hamiltonian $H_1$ and the same initial
data. Applying this with $H_1=H_{\mathrm{unc}}$ and $H_2=H$, we obtain
\[
  G_+^p(x,t)\le G_{\mathrm{unc}}^p(x,t).
\]
Xin and Yu proved that, in the parameter range
\eqref{eq:parameter-window}, there is a particular spatial location
$x_{\mathrm f}$ for which
\[
  \liminf_{t\to\infty}
  -\frac{G_{\mathrm{unc}}^p(x_{\mathrm f},t)}{t}
  \ge C\frac{A}{\log A}>0.
\]
Since $G_+^p\le G_{\mathrm{unc}}^p$, we also have
\[
  -\frac{G_+^p(x_{\mathrm f},t)}{t}
  \ge
  -\frac{G_{\mathrm{unc}}^p(x_{\mathrm f},t)}{t}.
\]
Therefore the same positive lower bound holds for the physical equation \eqref{eq:cutoff-main}:
the positive part cutoff does not destroy the fast propagation behavior
at $x_{\mathrm f}$.

\textit{2. Construct a Convex Upper Bound for the Physical Hamiltonian.}
The first step gives us a location where the flame propagates rapidly.
To prove that there is no single constant long run propagation speed
describing the whole flame, we need a different behavior at another
microscopic phase.  We use a horizontal channel through
$x_{\mathrm s}=(\pi,0)$, where the gas velocity vanishes,
$V_A(x_{\mathrm s})=0$; such a point is called a \emph{stagnation
point}.  We prove that the physical solution $G_+^p$ remains bounded
below on the entire channel for all time. To obtain this lower bound, we replace the part of the Hamiltonian
describing the flame's burning part by a simpler expression that
is always at least as large.
For the cellular flow \eqref{cellular_flow}, equation  \eqref{eq:strain-tensor-intro} gives
\[
  S_{V_A}(x)
  =
  A
  \begin{pmatrix}
    -\Phi(x)&0\\
    0&\Phi(x)
  \end{pmatrix},
\]
so the burning part of the physical Hamiltonian can be written as
\[
  B_+(x,q)
  =
  \frac{
    \bigl[
      (1-Ad\Phi(x))q_1^2
      +(1+Ad\Phi(x))q_2^2
    \bigr]_+
  }{|q|}.
\]
Define
\[
  a_1(x)\coloneqq (1-Ad\Phi(x))_+,
  \qquad
  a_2(x)\coloneqq (1+Ad\Phi(x))_+.
\]
Using
\[
  [r+s]_+\le r_++s_+
  \qquad\text{and}\qquad
  \frac{q_i^2}{|q|}\le |q_i|,
\]
we obtain the pointwise upper bound
\begin{equation}\label{upper-bd}
    B_+(x,q)
  \le
  a_1(x)|q_1|+a_2(x)|q_2|.
\end{equation}
Since the transport term $A V_0(x)\cdot q$ is unchanged, the upper bound \eqref{upper-bd} on $B_+$ gives a pointwise upper bound on the full physical
Hamiltonian,
\[
  H(x,q)\le \widehat H(x,q)\coloneqq 
  A V_0(x)\cdot q
  +a_1(x)|q_1|
  +a_2(x)|q_2|.
\]

There is a useful geometric interpretation of this expression.  At each
point $x$, consider the rectangle
\[
  \mathcal K(x)
  \coloneqq 
  A V_0(x)
  +[-a_1(x),a_1(x)]e_1
  +[-a_2(x),a_2(x)]e_2
\]
The vectors in $\mathcal K(x)$ are the local velocities allowed by an
auxiliary comparison model that we construct for the proof. They consist
of the actual gas velocity $A V_0(x)$ together with independently
variable horizontal and vertical corrections of sizes at most $a_1(x)$
and $a_2(x)$. Then,
\[
  \widehat H(x,q)
  =
  \sup_{v\in\mathcal K(x)} v\cdot q.
\]
The right hand side is called the \emph{support function} of
$\mathcal K(x)$: for a given $q$, it takes the largest value of the dot
product $v\cdot q$ among all $v\in\mathcal K(x)$.

Equivalently, following the credal set tradition initiated by Levi
\cite{Levi1980}, associate with $\mathcal K(x)$ the credal set
\[
  \mathfrak P_x
  \coloneqq 
  \overline{\operatorname{co}}
  \{\delta_v:v\in\mathcal K(x)\}.
\]
Here a credal set is understood as a weakly closed convex set of
probability laws
\cite{AugustinEtAl2014,TroffaesDeCooman2014}.
Since $\mathcal K(x)$ is a closed and bounded rectangle in $\mathbb R^2$
(up to translation by the fixed vector $A V_0(x)$), it is compact. In turn, $\mathfrak P_x$ is precisely the
compact convex set of all probability measures supported on
$\mathcal K(x)$. Recall that, for a generic observable $Z$, its upper expectation (or upper prevision \cite{AugustinEtAl2014,Walley1991}) with respect to $\mathfrak P_x$ is $\overline{\mathbb E}_{\mathfrak P_x}[Z]=\sup_{P \in \mathfrak P_x} \int Z \mathrm d P$. Then,
\[
  \widehat H(x,q)
  =
  \sup_{P\in\mathfrak P_x}
  \int v\cdot q \mathrm d P(v) = \overline{\mathbb E}_{\mathfrak P_x}[v\cdot q],
\]
so $\widehat H(x,q)$ is the upper expectation of the observable
$v\mapsto v\cdot q$ over the credal set $\mathfrak P_x$. 
At each fixed $x$, this is also a state-dependent instance of Peng's
worst case mean uncertainty construction, see
\Cref{impr-prob-sec}. This probabilistic interpretation is not merely terminology. The set $\mathcal K(x)$ has a rectangular product structure: its
horizontal and vertical corrections can be chosen independently within
their respective intervals.  In the associated control problem, the admissible controls may also be
chosen anew over time and concatenated, so the resulting uncertainty
model is dynamically rectangular, in the sense familiar from robust
dynamic programming \cite{Iyengar2005}, and is compatible with the usual
dynamic programming principle
\cite[Chapter~III, Section~3.1]{BardiCapuzzoDolcetta1997}.  This is the
same time consistency mechanism that rectangular ambiguity provides in
imprecise probability and robust control.

The resulting Hamiltonian $\widehat H$ is convex in $q$ and admits a
standard deterministic control representation established in
Proposition \ref{prop:control-representation}.  Most importantly, this
auxiliary control problem contains a one dimensional region from which
none of its admissible trajectories can escape.  This trapped region,
constructed in the next step, will give the lower bound on $G_+^p$ that
we need at the stagnation point.
Let $\widehat G^p$ denote the solution of the Hamilton-Jacobi equation
with Hamiltonian $\widehat H$ and the same initial data as $G_+^p$.
Since
  $H(x,q)\le \widehat H(x,q)$,
the comparison principle gives
  $\widehat G^p(x,t)\le G_+^p(x,t)$.
Thus any lower bound proved for the simpler solution $\widehat G^p$ is
automatically a lower bound for the physical solution $G_+^p$. The operative PDE device is the support function majorant, while its
upper expectation representation identifies the same construction with
rectangular credal uncertainty. This exact correspondence allows the
counterexample to be interpreted simultaneously in Hamilton-Jacobi,
robust control, and imprecise probability terms.

\textit{3. Trap the Auxiliary Dynamics in the Maximal Horizontal Channel.}
We now exploit the deterministic control interpretation of
$\widehat H$.  Using
\[
  |q_i|=-\min_{|\eta_i|\le1}\eta_iq_i,
\]
the reversed controlled dynamics are
\[
  \dot Y(t)
  =-A V_0(Y(t))
  +a_1(Y(t))\eta_1(t)e_1
  +a_2(Y(t))\eta_2(t)e_2,
  \qquad |\eta_i(t)|\le1.
\]
Set
\[
  \mu=Ad,
  \qquad
  r_\mu=\arccos\frac1\mu,
  \qquad
  \mathcal D_\mu=[\pi-r_\mu,\pi+r_\mu]\times\{0\}.
\]
The lower bound in \eqref{eq:parameter-window} gives $\mu>1$, so
$r_\mu\in(0,\pi/2)$.  At $(\pi+z,0)$ with $|z|\le r_\mu$,
\[
  a_2(\pi+z,0)=(1-\mu\cos z)_+=0,
  \qquad
  V_{A,2}(\pi+z,0)=0.
\]
Thus both the transverse drift and the coefficient of the transverse
control vanish throughout the channel, so a trajectory starting there
cannot leave it vertically.

At the endpoints,
\[
  a_1(\pi\pm r_\mu,0)=2,
  \qquad
  V_{A,1}(\pi\pm r_\mu,0)=\pm A\sin r_\mu.
\]
All admissible endpoint velocities point into the channel precisely when
\[
  A\sin r_\mu>2.
\]
Since $A=\mu/d$ and
$\sin r_\mu=\sqrt{\mu^2-1}/\mu$, this is equivalent to
\[
  \mu^2>1+4d^2,
\]
which is exactly the lower condition in
\eqref{eq:parameter-window}.  Hence every admissible trajectory starting
at any $x\in\mathcal D_\mu$ remains in $\mathcal D_\mu$ for all time.
The control representation then gives
\[
  \widehat G^p(x,t)
  \ge\min_{y\in\mathcal D_\mu}p\cdot y
  =\pi p_1-r_\mu|p_1|,
  \qquad x\in\mathcal D_\mu,\quad t\ge0.
\]

\textit{4. Combine the Two Comparisons.}
For every unit slope $p$, the two comparisons give
\[
  \widehat G^p\le G_+^p\le G_{\mathrm{unc}}^p.
\]
The left inequality transfers the whole channel lower bound to the
physical solution, while the right inequality transfers the Xin-Yu
fast point estimate at $x_{\mathrm f}$.  These two fixed microscopic
phases force the periodic correction to have oscillation growing at
least at rate $CA/\log A$.  The same comparison is stable under bounded
periodic perturbations of the planar initial datum.

\textit{Relation to the uncut strain equation.}
The closest predecessor to our approach is the work of Xin and Yu
\cite{XinYu2014}, who studied the same cellular flow for the strain
equation without the positive part cutoff and found a regime in which
fast propagation and local stagnation coexist. Their result does not
imply ours.  Indeed, the inequality between the two Hamiltonians
transfers their fast point estimate to the physical model, but the
corresponding comparison of solutions goes in the wrong direction for
transferring the stagnant point estimate.  The new ingredient in our proof is precisely the comparison in the
opposite direction: a convex upper bound on the physical Hamiltonian
whose associated control problem traps trajectories in a bounded
horizontal segment.  This supplies the missing stagnation estimate and
shows that the nonhomogenizing regime survives the physical cutoff.

Writing $\Phi(x)=\cos x_1\cos x_2$, Xin, Yu, and Ronney also suggest
first considering the simplified Hamiltonian
\[
  \widetilde H_{\mathrm{XYR}}(x,q)
  =
  \bigl(|q|-\mu\Phi(x)(|q_1|-|q_2|)\bigr)_+
  +V_A(x)\cdot q
\]
and then invoking comparison
\cite[Section~5.2]{XinYuRonney2024}.  This Hamiltonian is not globally
ordered with $H_+$.  Indeed,
\[
  \frac{q_1^2-q_2^2}{|q|}
  =
  (|q_1|-|q_2|)
  \frac{|q_1|+|q_2|}{|q|},
  \qquad q\ne0.
\]
When $q_1q_2\ne0$ and $|q_1|\ne|q_2|$, the last factor is strictly
larger than one.  Since $\Phi$ takes both signs arbitrarily close to
zero, one can keep both positive parts active while reversing the order
according to the sign of $\Phi(x)(|q_1|-|q_2|)$.  Thus neither
$\widetilde H_{\mathrm{XYR}}\le H_+$ nor
$H_+\le\widetilde H_{\mathrm{XYR}}$ holds globally.  The role of
$\widehat H$ is precisely to provide a genuine global upper bound while
retaining a convex support function structure and an exactly vanishing
transverse coefficient on the comparison channel.

\subsection{The Proof of \Cref{thm:no barrier} in Three Moves}

The proof has a different structure from the cellular counterexample.
It combines a local consequence of the physical cutoff with the global
volume preservation imposed by incompressibility. Fix a nonempty closed
set $D\subset\T_L^n$ satisfying the assumptions of
\Cref{thm:no barrier}.

\textit{1. Convert the Barrier Condition into Motion Toward $D$.}
Take a point $y\notin D$ whose distance from $D$ is
  $r\coloneqq\dist(y,D)$.
Let $x\in D$ be a nearest point and write
\[
  y=x+r\nu,
  \qquad |\nu|=1.
\]
The vector $\nu$ is a proximal outward normal to $D$ at $x$. Set
\[
  s\coloneqq
  \nu^{\mathsf T}S_V(x)\nu
  =
  \nu^{\mathsf T}DV(x)\nu.
\]
The assumed barrier inequality gives
\[
  V(x)\cdot\nu
  \le
  \varepsilon-(1+ds)_+.
\]
A Taylor expansion from $x$ to $y$ then gives
\[
  V(y)\cdot\nu
  \le
  \varepsilon-(1+ds)_++rs+C_2r^2,
  \qquad
  C_2\coloneqq\frac12\|D^2V\|_\infty.
\]

The key elementary consequence of the positive part cutoff is
\[
  -(1+ds)_++rs\le-\frac rd,
  \qquad 0<r<d,
\]
for every real number $s$. Choose $r_0<\min\{d,L/2\}$ sufficiently
small that
\[
  C_2r_0\le\frac1{2d}.
\]
For every $0<r<r_0$, we therefore obtain
\[
  V(y)\cdot\nu
  \le
  \varepsilon-\frac r{2d}.
\]
The vector $\nu$ points away from $D$. Hence, whenever
$r>2d\varepsilon$, the component of the fluid velocity pointing away
from $D$ is negative. The forward fluid flow therefore moves $y$
closer to $D$.

\textit{2. Show that the Distance from $D$ Contracts.}
Let $\Phi_t$ denote the flow generated by $V$, and define $\rho(t)\coloneqq\dist(\Phi_t(y),D)$. 
Although the nearest point in $D$ may change with time, the upper
derivative of the distance is controlled by the velocity in a nearest
outward normal direction. The estimate from the first step therefore
gives
\[
  \Dplus\rho(t)
  \le
  \varepsilon-\frac{\rho(t)}{2d}
\]
whenever $0<\rho(t)<r_0$. Scalar comparison yields
\[
  \rho(t)
  \le
  2d\varepsilon+
  \bigl(\rho(0)-2d\varepsilon\bigr)_+
  e^{-t/(2d)}.
\]
Thus the distance from $D$ decreases toward the error scale
$2d\varepsilon$.

For $s\ge0$, write
  $D_s\coloneqq
  \{y\in\T_L^n:\dist(y,D)\le s\}$.
If
  $2d\varepsilon<r<r_0$,
the distance estimate implies
\[
  \Phi_t(D_r)\subset D_{r_t},
  \qquad
  r_t
  =
  2d\varepsilon+
  (r-2d\varepsilon)e^{-t/(2d)}.
\]
Since $r_t$ decreases to $2d\varepsilon$, the flow carries the whole
$r$ neighborhood of $D$ into successively smaller neighborhoods
approaching $D_{2d\varepsilon}$.

\textit{3. Use Incompressibility to Rule Out a Larger Hole.}
Suppose, for contradiction, that
  $R(D)>2d\varepsilon$.
Choose
\[
  2d\varepsilon
  <
  r
  <
  \min\{R(D),r_0\}.
\]
Because $r<R(D)$, the region between the two neighborhoods
$D_{2d\varepsilon}$ and $D_r$ contains a nonempty open set. Consequently,
\[
  |D_{2d\varepsilon}|<|D_r|.
\]

On the other hand, incompressibility means that the flow $\Phi_t$
preserves Lebesgue measure. Together with the inclusion from the
second step, this gives
\[
  |D_r|
  =
  |\Phi_t(D_r)|
  \le
  |D_{r_t}|.
\]
Letting $t\to\infty$, so that $r_t\downarrow2d\varepsilon$, yields
\[
  |D_r|
  \le
  |D_{2d\varepsilon}|,
\]
which contradicts the strict inequality above. Therefore
  $R(D)\le2d\varepsilon$.
When $\varepsilon=0$, every point of the torus has distance zero from
$D$, so $D=\T_L^n$. Applying the same estimate to
$\varepsilon_j\downarrow0$ gives
\Cref{cor:barriers-dense}.

\subsection{Relevance to Machine Learning and Statistics}

Because our work intersects with machine learning and statistics, we
briefly discuss its relevance to sequential decision making under model
uncertainty. Robust Markov decision processes optimize against
uncertainty in transition dynamics, often represented through ambiguity
sets \cite{Iyengar2005,NilimElGhaoui2005,WiesemannKuhnRustem2013}.
Rectangular ambiguity is especially important because it makes Bellman
recursion time consistent. It is tempting, however, to infer from time
consistency an unjustified form of long run state independence.

Our results isolate the distinction. The local credal sets are compact
and convex, while the associated uncertainty model is dynamically
rectangular and governed by a valid dynamic programming principle.
Nevertheless, the resulting long horizon value remembers where it
started: one state lies in a robust trap, while another exhibits
propagation at a positive linear rate. In average reward language \cite[Chapters~8-9]{Puterman1994},
rectangular robustification can therefore yield a valid Bellman recursion
without forcing the control problem to forget its initial state or settle
on a single universal average rate.

There is a parallel statistical interpretation. Ambiguity and credal
sets are also used to represent classes of plausible distributions or
Bayesian inputs, from robust Bayesian analysis to distributionally robust
inference
\cite{BergerRiosInsuaRuggeri2000,DuchiGlynnNamkoong2021}.
Constructing or statistically calibrating such a set addresses which
laws are regarded as plausible; it does not by itself ensure that
repeated propagation through the resulting model forgets its initial
condition. That is a separate stability question. Correspondingly,
asymptotic results for robust procedures under dependent data invoke
mixing or ergodicity conditions, while ergodicity of imprecise Markov
chains governs convergence toward initial state independent long run
behavior
\cite{DeCoomanHermansQuaeghebeur2009,DuchiGlynnNamkoong2021,HermansDeCooman2012}.

The common lesson is structural. For machine learning, a time-consistent
robust Bellman operator does not by itself justify a state-independent
average reward under distribution shift. For statistics, a class of
plausible laws does not by itself imply asymptotic forgetting in a
sequential robust procedure. In both cases, one needs an additional
global stability or ergodicity mechanism. Our counterexample provides a
smooth continuous state system in which the local uncertainty model and
dynamic programming are well behaved, yet a state-independent long run
rate still fails.

The barrier certificate theorem adds a complementary lesson for robust
learning and statistical certification.  In robust sequential decision
making, one may seek a closed set $D$ carrying the local outward-normal
certificate
$H(x,\nu)\le\varepsilon$.  \Cref{thm:no barrier} shows that a
certificate of this precise type cannot remain genuinely localized:
$R(D)\le2d\varepsilon$, and an exact certificate forces
$D=\T_L^n$.  Thus the system can remember where it started even though
its state dependence is not explained by a proper exact outward barrier
certificate.  The theorem does not rule out every other global
invariance or blocking mechanism.

There is also a conditional statistical reading. Suppose that a
data-driven estimate $H_N^{\mathrm{est}}$ satisfies
\[
  \left|H(x,\nu)-H_N^{\mathrm{est}}(x,\nu)\right|
  \le \delta_N
\]
uniformly over all pairs $(x,\nu)$ with $x\in D$ and $\nu$ a unit
proximal outward normal of $D$ at $x$, and suppose that
  $H_N^{\mathrm{est}}(x,\nu)\le\tau_N$
on those pairs. If
$0\le\tau_N+\delta_N\le\varepsilon_0$, then
\Cref{thm:no barrier} yields
\[
  R(D)\le2d(\tau_N+\delta_N).
\]
Hence a uniform estimation or calibration error is converted directly
into a geometric coverage error. The theorem supplies this deterministic
transfer step; obtaining a probabilistic bound on $\delta_N$ requires
additional sampling assumptions.

\section{Related Literature}\label{sec:literature}
In this Section, we review the main strands of literature on which our
analysis builds. As our proof combines techniques from combustion,
Hamilton-Jacobi homogenization, robust control, and imprecise
probability, we discuss relevant results from these areas in
Sections~\ref{flame-sec}, \ref{H-J-sec}, \ref{rob-control-sec}, and
\ref{impr-prob-sec}, respectively.

\subsection{$G$-equations, Flame Stretch, and Effective Speeds}\label{flame-sec}
The $G$-equation was formally introduced by Williams
\cite{Williams1985}, with an earlier form already appearing in
Markstein's work \cite{Markstein1964}; see also Peters
\cite{Peters2000} for the combustion background. The level set methodology was placed in a systematic Hamilton-Jacobi
and numerical framework by Osher and Sethian \cite{OsherSethian1988}.  The recent survey of Xin, Yu, and Ronney \cite{XinYuRonney2024} gives a broad account of the physical models, the effective burning velocity problem, and the interaction between PDE and Lagrangian methods.

For the inviscid equation with constant laminar speed, periodic homogenization for incompressible or nearly incompressible flows was established by Xin and Yu \cite{XinYu2010} and by Cardaliaguet, Nolen, and Souganidis \cite{CardaliaguetNolenSouganidis2011}.  Quantitative enhancement laws depend strongly on the flow geometry.  Cellular flows produce the characteristic $A/\log A$ growth scale, while shear and three dimensional flows exhibit different mechanisms.

Flame stretch introduces curvature and strain corrections.  Numerical work by Liu, Xin, and Yu \cite{LiuXinYu2013} documented enhancement, slowdown, local quenching, and global quenching in cellular flows.  The uncut strain equation was analyzed rigorously by Xin and Yu \cite{XinYu2014}.  They found three regimes as the flow intensity increases: homogenization below a threshold, nonhomogenization with coexisting fast and stagnant regions just above the threshold, and global trapping for sufficiently large intensity.  They explicitly identified the positive part physical correction as a
separate model for future study.  Although the Hamiltonian inequality
$\Hunc\leq\Hcut$ transfers their fast point estimate to the physical
model, the induced solution inequality $\Gcut\leq\Gunc$ cannot transfer
their stagnant point lower bound.  The present paper supplies the
missing comparison and shows that the positive part correction does not
remove the intermediate nonhomogenizing regime.

For the physical curvature equation, Gao, Long, Xin, and Yu
\cite{GaoLongXinYu2024} proved existence of an effective burning
velocity for the standard two dimensional cellular flow and discussed
extensions of their method to more general two dimensional
incompressible periodic flows.  Mitake, Mooney, Tran, Xin, and Yu \cite{MitakeMooneyTranXinYu2025} later proved a bifurcation between homogenization and nonhomogenization for higher dimensional shear flows.  Their result shows that a physical cutoff is not, by itself, a universal homogenizing mechanism for geometric front equations.  Our theorem establishes the analogous failure for the first order physical strain equation, in the two dimensional cellular setting.

\subsection{Noncoercive and Nonconvex Hamilton-Jacobi Homogenization}\label{H-J-sec}
Classical periodic homogenization is particularly well developed for
coercive Hamiltonians, which grow uniformly in $x$ as $|q|\to\infty$,
because coercivity provides the compactness needed in the homogenization
argument and supports the standard periodic cell problem approach.
Noncoercive equations require additional geometric, dynamical, or
controllability hypotheses. Barles \cite{Barles2007} obtained homogenization under structural
conditions that recover sufficient control of the solution oscillations
despite the lack of full coercivity. Bardi and Terrone \cite{BardiTerrone2013} studied noncoercive
Hamilton-Jacobi-Isaacs equations, proving homogenization under
partial coercivity and asymptotic controllability hypotheses, while also
describing classes for which pointwise homogenization fails.

State-dependent long time behavior is likewise possible in deterministic
periodic noncoercive models.  Cardaliaguet \cite{Cardaliaguet2010} studied a two dimensional
noncoercive, nonconvex Hamiltonian, showing that under suitable
nonresonance conditions the long time average is constant, whereas
when these conditions fail it can retain a dependence on position.
Cardaliaguet, Lions, and Souganidis \cite{CardaliaguetLionsSouganidis2009}
analyzed moving interface laws with either oscillatory first order
normal velocity or curvature-dependent velocity coupled with oscillatory
forcing, while Caffarelli and Monneau \cite{CaffarelliMonneau2014}
proved two dimensional homogenization and constructed
higher dimensional counterexamples for forced mean curvature motion in
periodic media. These precedents establish that
noncoercivity can preserve microscopic information.  They do not,
however, address the special coupling between incompressible transport,
the strain tensor, and the physical positive part cutoff that we study in this paper.

The strain Hamiltonian is rather singular even within this literature:
it is positively homogeneous, direction-dependent, and nonconvex, while
its burning term may vanish in some normal directions even when the
drift remains large.  The counterexample we find exploits exactly this directional degeneracy.  A recent complementary development is the differential game approach of
Davini, Saona, and Ziliotto \cite{DaviniSaonaZiliotto2026} to stochastic
homogenization for classes of nonconvex and noncoercive Hamiltonians
under a uniform orientation condition on the dynamics. Our example
illustrates why some global mechanism preventing persistent trapping
or phase dependence is needed to recover homogenization.

\subsection{Differential Games, Robust Control, and Hamiltonian Comparisons}
\label{rob-control-sec}
Nonconvex Hamiltonians arise naturally as values of zero-sum differential games \cite{EvansSouganidis1984,Isaacs1965}.  The dynamic programming and viscosity solution framework is developed in Bardi and Capuzzo-Dolcetta \cite{BardiCapuzzoDolcetta1997}.  For our counterexample, a full two player representation is unnecessary.  We instead place the cutoff Hamiltonian below the support function of a rectangular velocity set.  This produces an auxiliary one player control problem whose solution can
be compared with that of the physical equation.

The direction of comparison is crucial. If $H_1\le H_2$ and the
corresponding Hamilton-Jacobi equations have the same initial data,
then the comparison principle gives $u_2\le u_1$. Thus, replacing the
physical Hamiltonian by a pointwise upper bound produces a lower bound
on its solution.  This is exactly what is needed at the stagnant point.  In contrast, the uncut Hamiltonian lies below the physical one and therefore gives the upper solution bound needed at the fast point.  The counterexample depends on using both comparisons at once.

\subsection{Imprecise Probability, Credal Dynamics, and Robust Learning}\label{impr-prob-sec}
Imprecise probability represents model uncertainty through lower and upper expectations or through credal sets of probability laws \cite{AugustinEtAl2014,TroffaesDeCooman2014,Walley1991}.  Upper transition operators propagate upper expectations when the
transition law is not known precisely but belongs to a set of possible
laws, with lower expectations obtained by conjugacy; iterating these
operators gives the dynamics of imprecise Markov chains
\cite{DeCoomanHermansQuaeghebeur2009,HermansDeCooman2012,TJoensDeBock2021}.
Their continuous time counterparts can be formulated through imprecise
Markov semigroups, whose long run and ergodic behavior has also been
studied \cite{CaprioChen2026}.

The same upper expectation calculus appears in distributionally robust optimization and robust dynamic programming.  Iyengar \cite{Iyengar2005}, Nilim and El Ghaoui \cite{NilimElGhaoui2005}, and Wiesemann, Kuhn, and Rustem \cite{WiesemannKuhnRustem2013} established foundational robust MDP formulations under transition ambiguity.  Recent credal machine learning methods use sets of priors, likelihoods, predictive distributions, or deployment laws to separate epistemic (reducible) from aleatoric (irreducible) uncertainty and to protect decisions under shift \cite{CaprioEtAl2024,CaprioSultanaEliaCuzzolin2024,ChenBerrettDamoulasCaprio2026}.  H\"ullermeier and Waegeman \cite{HuellermeierWaegeman2021} survey the broader distinction between aleatoric and epistemic uncertainty in machine learning.

The connection between our construction and imprecise probability is
exact: the auxiliary Hamiltonian $\widehat H$ is the upper expectation
of the observable $v\mapsto v\cdot q$ over a local credal set of
probability laws supported on the velocity set $\mathcal K(x)$. Our counterexample then shows that even when
these local credal sets are convex and compact and the associated
uncertainty model is dynamically rectangular, the long run robust value
need not forget its initial state. Additional global structure is needed
to ensure that different initial states share the same asymptotic
behavior.

\textit{Relation to Peng's sublinear-expectation framework.}
The $G$ in the combustion $G$-equation should not be identified with
the nonlinear generator $G$ in Peng's $G$-expectation. Nevertheless,
the auxiliary Hamiltonian $\widehat H$ has a precise first order
connection with Peng's framework. For each fixed $x$, let
$\xi_x(v)=v$ on $\mathcal K(x)$ and define, for
$\varphi\in C(\mathcal K(x))$,
\[
  \mathcal E_x[\varphi(\xi_x)]
  \coloneqq
  \sup_{v\in\mathcal K(x)}\varphi(v)
  =
  \sup_{P\in\mathfrak P_x}
  \int \varphi(v)\,\mathrm dP(v).
\]
Then,
\[
  \widehat H(x,q)
  =
  \mathcal E_x[\xi_x\cdot q].
\]
This is a state-dependent version of Peng's worst case distribution for
mean uncertainty: the first order part of Peng's nonlinear generator \cite{Peng2008}
has the support function form
\[
  G_{\mathrm{Peng}}(q,0)
  =
  \sup_{\theta\in\Theta}\theta\cdot q.
\]
The full $G_{\mathrm{Peng}}$ generator also contains a second order
variance or covariance uncertainty term, leading to the $G$-heat
equation, $G$-Brownian motion, and the associated stochastic calculus
\cite{Peng2007}. Our auxiliary model instead contains state dependent
velocity uncertainty and a first order Hamilton-Jacobi equation.
Moreover, because the physical cutoff Hamiltonian $H(x,\cdot)$ is
generally nonconvex, the exact sublinear expectation identification
applies to the convex majorant $\widehat H$, not to $H$ itself. Thus the
connection with Peng's theory is exact at the level of local sublinear
expectations and first order mean uncertainty, but it is not a direct
application of $G$-Brownian motion or the $G$-It\^o calculus.

Viewed through this lens, the two main results separate asymptotic state
dependence from the outward certificate geometry studied here: the
former occurs in \Cref{thm:main}, whereas \Cref{thm:no barrier} shows
that every exact outward barrier certificate for the physical
Hamiltonian fills the torus and forces increasingly accurate
certificates to become dense.

\section{Setting and the Hamiltonian Sandwich}\label{sec:setting}
We now turn to the formal proof. The main ideas and the role of the two
comparison arguments were explained in detail in \Cref{proof-overview}; here we introduce the precise notation and record the ingredients
needed to implement that strategy. In particular, we define the physical
and uncut Hamiltonians and establish the first half of the solution
sandwich, while the second half will follow from the credal construction
in \Cref{sec:credal}.

\subsection{The Cellular Flow}
Throughout the main proof we work on the $2\pi$-periodic torus $\T^2_{2\pi}$.  Set
\begin{equation}\label{eq:V0}
  V_0(x)=\bigl(-\sin x_1\cos x_2, \cos x_1\sin x_2\bigr),
  \qquad
  V_A=A V_0.
\end{equation}
The field is divergence free because
\[
  \partial_{x_1}(-\sin x_1\cos x_2)
  +\partial_{x_2}(\cos x_1\sin x_2)=0.
\]
A direct computation gives \eqref{eq:strain-tensor-intro}.
Let $\mu\coloneqq Ad$ and define
\[
  \Psi(q)
  \coloneqq 
  \begin{cases}
  \dfrac{q_1^2-q_2^2}{|q|},&q\ne0,\\[0.8em]
  0,&q=0.
  \end{cases}
\]
The uncut and physical-cutoff Hamiltonians can then be written on all of
$\mathbb R^2$ as
\begin{align}
  \Hunc(x,q)
  &\coloneqq |q|-\mu\Phi(x)\Psi(q)+A V_0(x)\cdot q,
  \label{eq:Hunc}\\
  \Hcut(x,q)
  &\coloneqq \bigl(|q|-\mu\Phi(x)\Psi(q)\bigr)_+
     +A V_0(x)\cdot q.\nonumber
\end{align}

\begin{lemma}[Regularity of the Uncut and Cutoff Hamiltonians]
\label{lem:unc-cut-regularity}
The Hamiltonians $\Hunc$ and $\Hcut$ are continuous, periodic in $x$,
positively homogeneous in $q$, and globally Lipschitz in $q$, uniformly
in $x$. Moreover, there exists a constant $C=C(A,d)>0$ such that
\[
  |H_i(x,q)-H_i(y,q)|
  \le C \dist(x,y)|q|,
  \qquad
  H_i\in\{\Hunc,\Hcut\}.
\]
In particular, both Hamiltonians satisfy the hypotheses of
\Cref{thm:comparison-background}.
\end{lemma}

\begin{proof}
For $i=1,2$, set
\[
  h_i(q)
  \coloneqq 
  \begin{cases}
  q_i^2/|q|,&q\ne0,\\
  0,&q=0.
  \end{cases}
\]
We first show that each $h_i$ is globally Lipschitz. Let $q,q'\in\mathbb R^2$. Assume, without loss of generality, that
$|q|\ge|q'|>0$. Then,
\begin{align*}
  |h_i(q)-h_i(q')|
  &\le
  \frac{|q_i^2-(q_i')^2|}{|q|}
  +(q_i')^2
  \left|\frac1{|q|}-\frac1{|q'|}\right|\\
  &\le 2|q-q'|+|q-q'|
  =3|q-q'|.
\end{align*}
If $q'=0$, the same estimate follows from
$0\le h_i(q)\le|q|$. Thus $\Psi=h_1-h_2$ is globally Lipschitz, with
Lipschitz constant at most $6$, and
\[
  |\Psi(q)|\le|q|.
\]

The map $r\mapsto r_+$ is $1$-Lipschitz. Since
$|\Phi|\le1$ and $|V_0|\le1$, both Hamiltonians are therefore globally
Lipschitz in $q$, uniformly in $x$. Finally, if $L_\Phi$ and $L_V$ are
Lipschitz constants for $\Phi$ and $V_0$ on the torus, then
\[
  |H_i(x,q)-H_i(y,q)|
  \le
  \bigl(\mu L_\Phi+A L_V\bigr)
  \dist(x,y)|q|,
\]
for $H_i=\Hunc,\Hcut$. The remaining claims follow directly.
\end{proof}

\subsection{Comparison for Planar Data}
For each $p\in\mathbb R^2$, let
  $G_{\mathrm{unc}}^p$
   and 
  $G_+^p$ denote the viscosity solutions associated with $\Hunc$ and $\Hcut$,
respectively, with common planar initial datum $p\cdot x$.

For a slope $p$, write $G_i(x,t)=p\cdot x+u_i(x,t)$, where $u_i$
is periodic. Standard viscosity comparison on the torus applies to
$u_i$; the conventions and order principle used below are recorded in
Appendix \ref{app:viscosity}.

\begin{lemma}[Order Reversal Between Hamiltonians and Solutions]\label{lem:comparison}
Let $H_1,H_2$ be continuous periodic Hamiltonians for which comparison holds, and suppose
\[
  H_1(x,q)\le H_2(x,q),
  \qquad\text{for all }x,q.
\]
If $G_i$ solves
\[
  (G_i)_t+H_i(x,DG_i)=0,
  \qquad
  G_i(x,0)=p\cdot x,
\]
then
\[
  G_2(x,t)\le G_1(x,t),
  \qquad\text{for all }x,t\ge0.
\]
\end{lemma}

\begin{proof}
The solution $G_2$ is a viscosity subsolution of the $H_1$ equation because
\[
  (G_2)_t+H_1(x,DG_2)
  \le (G_2)_t+H_2(x,DG_2)=0.
\]
Subtract $p\cdot x$ and apply periodic comparison with the common initial datum $0$.
\end{proof}

\begin{proposition}[Quantitative Phase Separation from a Hamiltonian Sandwich]
\label{prop:phase-separation-sandwich}
Let $n\ge1$ and $L>0$, and let
\[
  H_{\mathrm{lo}},H,H_{\mathrm{hi}}:
  \T_L^n\times\mathbb R^n\to\mathbb R
\]
be continuous Hamiltonians, periodically lifted in the state variable,
for which comparison holds in the planar periodic class.  Assume
\[
  H_{\mathrm{lo}}(x,q)\le H(x,q)\le H_{\mathrm{hi}}(x,q),
  \qquad\text{for all }x,q.
\]
Fix $p\in\mathbb R^n$, and let
$G_{\mathrm{lo}}^p,G^p,G_{\mathrm{hi}}^p$ be the corresponding
viscosity solutions on $\mathbb R^n$ with common planar initial datum
$p\cdot x$.  Suppose that there exist
$x_{\mathrm s},x_{\mathrm f}\in\mathbb R^n$,
a constant $B\in\mathbb R$, and a constant $c>0$ such that
\[
  G_{\mathrm{hi}}^p(x_{\mathrm s},t)\ge B,
  \qquad\text{for all }t\ge0,
\]
and
\[
  \liminf_{t\to\infty}
  -\frac{G_{\mathrm{lo}}^p(x_{\mathrm f},t)}{t}
  \ge c.
\]
Set $u^p(x,t)=G^p(x,t)-p\cdot x$.  Then,
\[
  G^p(x_{\mathrm s},t)\ge B,
  \quad\text{for all }t\ge0,
  \qquad
  \liminf_{t\to\infty}
  -\frac{G^p(x_{\mathrm f},t)}{t}
  \ge c,
\]
and
\begin{equation}\label{eq:sandwich-oscillation}
  \liminf_{t\to\infty}
  \frac{\osc_{\T_L^n}u^p(\cdot,t)}{t}
  \ge c.
\end{equation}
Consequently, $G^p$ has no phase-independent long-time rate.  In the
$G$-equation setting, this excludes the existence of an effective
burning velocity.
\end{proposition}

\begin{proof}
By \Cref{lem:comparison}, the Hamiltonian inequalities imply
\[
  G_{\mathrm{hi}}^p\le G^p\le G_{\mathrm{lo}}^p.
\]
The first two conclusions follow directly from the assumed bounds.  For
the oscillation, one has
\begin{align*}
  \osc_{\T_L^n}u^p(\cdot,t)
  \ge u^p(x_{\mathrm s},t)-u^p(x_{\mathrm f},t)\ge B-p\cdot x_{\mathrm s}
       -G_{\mathrm{lo}}^p(x_{\mathrm f},t)
       +p\cdot x_{\mathrm f}.
\end{align*}
Dividing by $t$ and taking the lower limit gives
\eqref{eq:sandwich-oscillation}.  A phase-independent effective rate
would force $\osc u^p(\cdot,t)/t\to0$, contradicting that estimate.
\end{proof}

Since $a_+\ge a$, then $\Hunc\le\Hcut$. Hence
\begin{equation}\label{eq:cut-below-unc-solution}
  G_+^p\le G_{\mathrm{unc}}^p,
  \qquad\text{for every }p\in\mathbb R^2.
\end{equation}
The second half of the sandwich will follow from the credal majorant constructed in \Cref{sec:credal}.

\section{The Fast Point From the Uncut Equation}\label{sec:fast}

The following is the precise result  from Xin and Yu that we use in our argument.

\begin{theorem}[Xin-Yu Fast Point Estimate {\cite[Theorem~1.3(ii)]{XinYu2014}}]
\label{thm:xinyu}
There exists a universal constant $C>0$ with the following property.  Suppose
\[
  0<d<\frac14,
  \qquad
  A\ge4,
  \qquad
  Ad\le1+\frac d{10}.
\]
Let $\Gunc^p$ solve the uncut strain equation with cellular flow \eqref{eq:V0} and initial datum $\Gunc^p(x,0)=p\cdot x$, where $p$ is a unit vector.  Then, at
\[
  x_{\mathrm f}=\left(\frac\pi2,0\right),
\]
one has
\[
  \liminf_{t\to\infty}
  -\frac{\Gunc^p(x_{\mathrm f},t)}t
  \ge C\frac A{\log A}.
\]
\end{theorem}

\Cref{thm:xinyu} has been stated in the notation and normalization
used in this paper. In particular, its cellular field, spatial period,
flow amplitude $A$, Markstein parameter $d$, planar initial datum, and
distinguished point $x_{\mathrm f}$ agree with
\eqref{eq:V0}, \eqref{eq:Hunc}, and the definitions above; no spatial
or temporal rescaling is used. Under \eqref{eq:parameter-window},
\[
  d<\frac14,
  \qquad
  A=\frac\mu d>\frac1d>\frac{399}{20}>4,
  \qquad
  \mu\le1+\frac d{10},
\]
so every hypothesis of \Cref{thm:xinyu} is satisfied with the same
universal constant $C$.

\begin{proposition}[The Cutoff Preserves the Fast Point]
\label{prop:fast-cutoff}
Under the hypotheses of \Cref{thm:main}, for every unit vector
$p\in\mathbb R^2$,
\[
  \liminf_{t\to\infty}
  -\frac{G_+^p(x_{\mathrm f},t)}t
  \ge C\frac A{\log A}.
\]
\end{proposition}

\begin{proof}
The Xin-Yu estimate in \Cref{thm:xinyu} holds for every unit vector $p$,
with the same universal constant $C$. By
\eqref{eq:cut-below-unc-solution},
\[
  -G_+^p(x_{\mathrm f},t)
  \ge -G_{\mathrm{unc}}^p(x_{\mathrm f},t).
\]
Then, it is enough to divide by $t>0$ and to apply \Cref{thm:xinyu}.
\end{proof}

\begin{remark}[Physical Meaning of the Comparison]
Where the strain-corrected local burning speed in the uncut model becomes
negative, the physical cutoff replaces it by zero. It therefore prevents
reversal of the flame's normal burning motion and, by the comparison
above, can only lower the level set function. Fast forward propagation
in the uncut model is consequently inherited by the physical model.
\end{remark}

\section{A Credal Upper Envelope}\label{sec:credal}
We now construct the second comparison needed in the proof: a convex
Hamiltonian $\widehat H$ that dominates the physical cutoff Hamiltonian.
It admits an exact interpretation as an upper expectation over a credal
set of probability laws supported on the state-dependent set
$\mathcal K(x)$.  For each planar slope $p$, the corresponding solution
$\widehat G^p$ lies below $G_+^p$ and provides the whole channel lower
bound needed for phase separation.
\subsection{A Coordinate Rectangular Envelope}
The next estimate identifies the part of the construction that is not
specific to the trigonometric cellular flow.

\begin{lemma}[Coordinate Rectangular Domination]
\label{lem:rectangular-domination}
Let $V:\T_L^2\to\mathbb R^2$ be $C^1$, let $d>0$, and fix an
orthonormal basis $(\tau,\nu)$ of $\mathbb R^2$.  Write
\[
  s_{\tau\tau}=\tau^{\mathsf T}S_V\tau,
  \qquad
  s_{\nu\nu}=\nu^{\mathsf T}S_V\nu,
  \qquad
  s_{\tau\nu}=\tau^{\mathsf T}S_V\nu,
\]
and define
\[
  b_\tau=(1+ds_{\tau\tau})_+ +d|s_{\tau\nu}|,
  \qquad
  b_\nu=(1+ds_{\nu\nu})_+ +d|s_{\tau\nu}|.
\]
Then, for every $q\in\mathbb R^2$,
\begin{equation}\label{eq:coordinate-majorant}
  \left(
    1+d\frac{q^{\mathsf T}S_V(x)q}{|q|^2}
  \right)_+|q|
  \le
  b_\tau(x)|q\cdot\tau|+b_\nu(x)|q\cdot\nu|,
\end{equation}
with both sides understood as zero at $q=0$.
\end{lemma}

\begin{proof}
For $q\ne0$, write
$q_\tau=q\cdot\tau$ and $q_\nu=q\cdot\nu$.  Then
\[
  q^{\mathsf T}(I+dS_V)q
  =(1+ds_{\tau\tau})q_\tau^2
   +2ds_{\tau\nu}q_\tau q_\nu
   +(1+ds_{\nu\nu})q_\nu^2.
\]
Using
\[
  [r+s+t]_+\le r_++|s|+t_+,
  \qquad
  \frac{q_i^2}{|q|}\le|q_i|,
  \qquad
  \frac{2|q_\tau q_\nu|}{|q|}
  \le|q_\tau|+|q_\nu|,
\]
and dividing by $|q|$ gives \eqref{eq:coordinate-majorant}.
\end{proof}

For the cellular flow, take $\tau=e_1$ and $\nu=e_2$.  By
\eqref{eq:strain-tensor-intro}, the mixed strain vanishes identically,
and \Cref{lem:rectangular-domination} gives
\[
  a_1(x)\coloneqq(1-\mu\Phi(x))_+,
  \qquad
  a_2(x)\coloneqq(1+\mu\Phi(x))_+.
\]
Define the convex Hamiltonian
\[
  \Hhat(x,q)
  \coloneqq A V_0(x)\cdot q+a_1(x)|q_1|+a_2(x)|q_2|.
\]
Then, $\Hcut\le\Hhat$.

\begin{lemma}[Regularity of the Rectangular Majorant]
\label{lem:hat-regularity}
The Hamiltonian $\Hhat$ is continuous, periodic in $x$, positively
homogeneous, convex, and globally Lipschitz in $q$, uniformly in $x$.
Moreover, there exists $C=C(A,d)>0$ such that
\[
  |\Hhat(x,q)-\Hhat(y,q)|
  \le C \dist(x,y)|q|.
\]
Consequently, $\Hhat$ satisfies the hypotheses of
\Cref{thm:comparison-background}.
\end{lemma}

\begin{proof}
The functions
\[
  a_1=(1-\mu\Phi)_+,
  \qquad
  a_2=(1+\mu\Phi)_+
\]
are Lipschitz and satisfy $0\le a_i\le1+\mu$. Hence
\[
  |\Hhat(x,q)-\Hhat(x,q')|
  \le
  \bigl[A+\sqrt2(1+\mu)\bigr]|q-q'|.
\]
If $L_\Phi$ and $L_V$ are Lipschitz constants for $\Phi$ and $V_0$,
then
\[
  |\Hhat(x,q)-\Hhat(y,q)|
  \le
  \bigl(A L_V+\sqrt2 \mu L_\Phi\bigr)
  \dist(x,y)|q|.
\]
Convexity and positive homogeneity follow directly from the formula for
$\Hhat$.
\end{proof}

By Lemma \ref{lem:rectangular-domination}, we have that 
  $\Hunc\le\Hcut\le\Hhat$.
For each $p\in\mathbb R^2$, let $\widehat G^p$ denote the viscosity
solution associated with $\Hhat$ and initial datum $p\cdot x$. For every $p\in\mathbb R^2$, Lemma \ref{lem:comparison} yields
  $\widehat G^p
  \le G_+^p
  \le G_{\mathrm{unc}}^p$.

\subsection{Upper Expectations and the Local Credal Set}
Set
\begin{equation}\label{eq:velocity-rectangle}
  \mathcal K(x)
  \coloneqq A V_0(x)
   +[-a_1(x),a_1(x)]e_1
   +[-a_2(x),a_2(x)]e_2.
\end{equation}
The set $\mathcal K(x)$ in \eqref{eq:velocity-rectangle} is a compact
convex rectangle translated by the fluid velocity.  Its support function is precisely
\[
  \sup_{v\in\mathcal K(x)}v\cdot q=\Hhat(x,q).
\]
Equivalently, with
\[
  \mathfrak P_x
  \coloneqq \clco\{\delta_v:v\in\mathcal K(x)\}
  \subset\Pp(\R^2),
\]
where the closure is taken in the weak topology, one has
  $\Hhat(x,q)
  =\upperE_{\mathfrak P_x}[v\cdot q]$.
Thus $\Hhat$ is the upper expectation of the linear observable
$v\mapsto v\cdot q$ over the local credal set $\mathfrak P_x$, whose
elements are probability laws supported on $\mathcal K(x)$. The replacement of the original direction-coupled burning term by the
rectangular majorant gives a pointwise upper bound on the physical
Hamiltonian, while the representation of $\Hhat$ through
$\mathfrak P_x$ is exact. Support functions, credal upper expectations, and the precise
deterministic meaning of rectangularity are summarized in
Appendices \ref{app:credal-background} and \ref{app:control-background}.

\subsection{Robust Control Representation}
Let $\mathcal U$ denote the set of measurable controls $\eta=(\eta_1,\eta_2):[0,\infty)\to[-1,1]^2$.  For $x\in\R^2$ and $\eta\in\mathcal U$, let $X^{x,\eta}$ solve
\begin{equation}\label{eq:majorant-dynamics}
  \dot X(t)
  =-A V_0(X(t))
   +a_1(X(t))\eta_1(t)e_1
   +a_2(X(t))\eta_2(t)e_2,
  \qquad
  X(0)=x.
\end{equation}
The coefficients are globally Lipschitz and periodic, so the lifted ODE has a unique Carath\'eodory solution.

\begin{proposition}[Control Representation]
\label{prop:control-representation}
For every $p\in\mathbb R^2$, the viscosity solution of
\[
  (\widehat G^p)_t+\Hhat(x,D\widehat G^p)=0,
  \qquad
  \widehat G^p(x,0)=p\cdot x,
\]
is
\begin{equation}\label{eq:value-representation}
  \widehat G^p(x,t)
  =
  \inf_{\eta\in\mathcal U}
  p\cdot X^{x,\eta}(t).
\end{equation}
\end{proposition}

\begin{proof}
Appendix~\ref{app:control} proves the formula directly on the lifted
space $\mathbb R^2$ for the more general initial datum
$g(x)=p\cdot x+u_0(x)$, where $u_0$ is continuous and periodic. The
argument does not apply a bounded data representation theorem to the
unbounded function $g$: bounded controlled speed gives finite terminal
values, while translation covariance of the trajectories shows that the
value function minus $p\cdot x$ is periodic. Periodic viscosity
uniqueness then identifies the value function with the solution in the
planar periodic class. Taking $u_0=0$ gives
\eqref{eq:value-representation}.
\end{proof}

The reversed drift in \eqref{eq:majorant-dynamics} is essential.  It
reflects the backward characteristic flow used to evaluate the solution
at a fixed spatial point.

\subsection{A General Comparison Channel Criterion}

\begin{proposition}[Comparison Channel Criterion]
\label{prop:channel-criterion}
Let $V:\mathbb R^2\to\mathbb R^2$ and
$b_1,b_2:\mathbb R^2\to[0,\infty)$ be bounded, Lipschitz, and periodic,
and define
\[
  H^{\Box}(x,q)
  =V(x)\cdot q+b_1(x)|q_1|+b_2(x)|q_2|.
\]
Let $H$ be a continuous periodic Hamiltonian satisfying
$H\le H^{\Box}$ and the comparison principle in the planar periodic
class.  Fix
\[
  \mathcal D=[x_0-r,x_0+r]\times\{y_0\}.
\]
Assume that, for every $|z|\le r$,
\begin{equation}\label{eq:channel-normal}
  V_2(x_0+z,y_0)=0,
  \qquad
  b_2(x_0+z,y_0)=0,
\end{equation}
and that, for some $\kappa>0$,
\begin{equation}\label{eq:channel-endpoints}
\begin{aligned}
  V_1(x_0+r,y_0)-b_1(x_0+r,y_0)&\ge\kappa,\\
  V_1(x_0-r,y_0)+b_1(x_0-r,y_0)&\le-\kappa.
\end{aligned}
\end{equation}
If $G_H^p$ denotes the solution with Hamiltonian $H$ and initial datum
$p\cdot x$, then
\[
  G_H^p(x,t)
  \ge\min_{y\in\mathcal D}p\cdot y,
  \qquad
  x\in\mathcal D,
  \quad t\ge0.
\]
\end{proposition}

\begin{proof}
The general support function representation in
Appendix~\ref{app:control-background} gives
\[
  G_{H^{\Box}}^p(x,t)
  =\inf_{\eta\in\mathcal U}p\cdot X^{x,\eta}(t),
\]
where
\[
  \dot X=-V(X)+b_1(X)\eta_1e_1+b_2(X)\eta_2e_2,
  \qquad |\eta_i|\le1.
\]
Condition \eqref{eq:channel-normal} makes the vertical component vanish
on $\mathcal D$.  At the right endpoint, the horizontal component is at
most $-\kappa$, and at the left endpoint it is at least $\kappa$, for
every control.  Hence \Cref{lem:segment-invariance-background} implies
that every trajectory starting in $\mathcal D$ remains there.  The
control representation gives
\[
  G_{H^{\Box}}^p(x,t)
  \ge\min_{y\in\mathcal D}p\cdot y.
\]
Finally, $H\le H^{\Box}$ implies
$G_{H^{\Box}}^p\le G_H^p$ by
\Cref{lem:H-order-background}.
\end{proof}

\section{The Cellular Comparison Channel}\label{sec:trap}

Let
  $\mu=Ad$, 
  $r_\mu=\arccos\frac1\mu$, and   
  $\mathcal D_\mu=[\pi-r_\mu,\pi+r_\mu]\times\{0\}$. 
The lower bound in \eqref{eq:parameter-window} implies $\mu>1$, so
$r_\mu\in(0,\pi/2)$.  At a point $(\pi+z,0)$ with
$|z|\le r_\mu$,
\[
  \Phi(\pi+z,0)=-\cos z,
\]
and therefore
\[
  a_2(\pi+z,0)=(1-\mu\cos z)_+=0,
  \qquad
  a_1(\pi+z,0)=1+\mu\cos z.
\]
Moreover, $V_{A,2}(x_1,0)=0$.  Thus the normal conditions
\eqref{eq:channel-normal} hold throughout $\mathcal D_\mu$.

At the two endpoints,
\[
  a_1(\pi\pm r_\mu,0)=2,
  \qquad
  V_{A,1}(\pi+r_\mu,0)=A\sin r_\mu,
  \qquad
  V_{A,1}(\pi-r_\mu,0)=-A\sin r_\mu.
\]
Since
\[
  A\sin r_\mu-2
  =\frac{\sqrt{\mu^2-1}}d-2>0
\]
by \eqref{eq:parameter-window}, the endpoint conditions
\eqref{eq:channel-endpoints} hold with
\[
  \kappa=\frac{\sqrt{\mu^2-1}}d-2.
\]
Applying \Cref{prop:channel-criterion} with
$H^{\Box}=\Hhat$, $V=V_A$, $b_1=a_1$, $b_2=a_2$,
$x_0=\pi$, and $y_0=0$, first with $H=\Hhat$ and then with
$H=\Hcut$, gives the whole channel estimates.

\begin{corollary}[Uniform Lower Bound on the Comparison Channel]
\label{cor:stagnation}
For every $p=(p_1,p_2)\in\mathbb R^2$,
\[
  \widehat G^p(x,t)
  \ge\min_{y\in\mathcal D_\mu}p\cdot y
  =\pi p_1-r_\mu|p_1|,
  \qquad
  x\in\mathcal D_\mu,
  \quad t\ge0,
\]
and hence
\[
  G_+^p(x,t)
  \ge\pi p_1-r_\mu|p_1|,
  \qquad
  x\in\mathcal D_\mu,
  \quad t\ge0.
\]
\end{corollary}

\section{Completion of the Proof of \Cref{thm:main}}\label{completion}
We now complete the proofs of \Cref{thm:main} and
\Cref{cor:periodic-perturbations}.

\begin{proof}[Proof of \Cref{thm:main}]
The whole channel estimate \eqref{eq:slow-main} is
\Cref{cor:stagnation}, and the fast estimate \eqref{eq:fast-main} is
\Cref{prop:fast-cutoff}.  Applying
\Cref{prop:phase-separation-sandwich} with
\[
  H_{\mathrm{lo}}=\Hunc,
  \qquad H=\Hcut,
  \qquad H_{\mathrm{hi}}=\Hhat,
\]
$x_{\mathrm s}=(\pi,0)$,
$B=\pi p_1-r_\mu|p_1|$, and $c=CA/\log A$ gives
\eqref{eq:oscillation-main}.  Indeed, the slow hypothesis for
$\widehat G^p$ is contained in \Cref{cor:stagnation}, and the fast
hypothesis is \Cref{thm:xinyu}.  Since a phase-independent effective
burning velocity would force
$\osc u^p(\cdot,t)/t\to0$, none exists.  The argument applies to
every unit direction $p$.
\end{proof}

\begin{proof}[Proof of \Cref{cor:periodic-perturbations}]
Since
\[
  p\cdot x+m\le p\cdot x+u_0(x)\le p\cdot x+M,
\]
order preservation and invariance under addition of constants for the
solution semigroup give \eqref{eq:periodic-data-sandwich}.  The slow
and fast estimates follow immediately.  At the two distinguished
phases,
\[
  \bigl(G_{+,u_0}^p(x_{\mathrm s},t)-p\cdot x_{\mathrm s}\bigr)
  -\bigl(G_{+,u_0}^p(x_{\mathrm f},t)-p\cdot x_{\mathrm f}\bigr)
\]
is bounded below by the corresponding planar-data difference plus
$m-M$.  Dividing by $t$ therefore gives the same lower bound in
\eqref{eq:oscillation-main}.
\end{proof}

\begin{remark}[Comparison Channel versus Outward Barrier Certificate]
\label{rem:not-barrier}
The segment $\mathcal D_\mu$ is invariant only for the reversed
controlled dynamics of the majorant $\Hhat$.  It is not an exact
outward barrier certificate for the physical Hamiltonian.  Indeed, at
the right endpoint $x=(\pi+r_\mu,0)$, whose outward normal is $e_1$,
\[
  \Hcut(x,e_1)=2+A\sin r_\mu>0.
\]
Thus the slow estimate is generated by an upper comparison channel, not
by an exact outward barrier certificate for the physical Hamiltonian.
The global rigidity statement in \Cref{thm:no barrier} makes this
distinction quantitative.
\end{remark}

\section{Rigidity of Outward Barrier Certificates}
\label{sec:no barrier}

We now prove \Cref{thm:no barrier} and
\Cref{cor:barriers-dense}.  Unlike the cellular counterexample, this
result holds for every $C^2$ incompressible periodic flow, in every
dimension.  It shows that a closed set satisfying the outward
Hamiltonian inequality in every proximal normal direction can be proper
only if it is quantitatively dense in the torus.

\begin{proof}[Proof of \Cref{thm:no barrier}]
If $D=\T_L^n$, the conclusion is immediate. We therefore assume that
$D$ is a proper subset of $\T_L^n$. Let
  $C_2\coloneqq 1/2\|D^2V\|_\infty$.
Choose $r_0>0$ such that
\[
  r_0<\min\left\{d,\frac L2\right\},
  \qquad
  C_2r_0\le\frac1{2d},
\]
with the second restriction omitted if $C_2=0$.  Choose $0<\e_0<r_0/(2d)$.
Take $y\notin D$ with
  $r\coloneqq \dist(y,D)<r_0$.
Let $x\in D$ be a nearest point and write, in compatible flat
coordinates,
\[
  y=x+r\nu,
  \qquad |\nu|=1.
\]
Then $\nu$ is a proximal outward normal. Set
\[
  s\coloneqq
  \nu^{\mathsf T}S_V(x)\nu
  =
  \nu^{\mathsf T}DV(x)\nu.
\]
The barrier condition gives
\[
  V(x)\cdot \nu
  \le\e-(1+ds)_+.
\]
Taylor's theorem yields
\begin{align*}
  V(y)\cdot \nu
  \le V(x)\cdot \nu+rs+C_2r^2\le\e-(1+ds)_++rs+C_2r^2.
\end{align*}
The elementary cutoff inequality
\begin{equation*}
  -(1+ds)_++rs\le-\frac rd,
  \qquad 0<r<d,
\end{equation*}
holds for every $s\in\R$.  Indeed, if $s\le-1/d$, the left side is $rs\le-r/d$; if $s>-1/d$, it equals $-1-(d-r)s$, whose supremum is $-r/d$.  Therefore,
\begin{equation}\label{eq:drift-app}
  V(y)\cdot \nu
  \le\e-\frac r{2d}.
\end{equation}

Let $\Phi_t$ be the flow generated by $V$, and set 
$\rho(t)\coloneqq\dist(\Phi_t(y),D)$. 
Whenever $0<\rho(t)<r_0$,
\Cref{lem:distance-Dini-background} and \eqref{eq:drift-app},
applied at the current point $\Phi_t(y)$, give
\[
  \Dplus\rho(t)
  \le\varepsilon-\frac{\rho(t)}{2d}.
\]
Let $I$ be a connected component of
$\{t\ge0:\rho(t)>0\}$. If $0\in I$, write $I=[0,\beta)$;
otherwise write $I=(\alpha,\beta)$, where $\alpha>0$ and continuity
gives $\rho(\alpha)=0$. On any portion of $I$ on which
$\rho<r_0$, scalar comparison for the upper Dini inequality gives
\[
  \rho(t)
  \le
  2d\varepsilon+
  \bigl(\rho(0)-2d\varepsilon\bigr)e^{-t/(2d)},
  \qquad 0\in I,
\]
and
\[
  \rho(t)
  \le
  2d\varepsilon
  \bigl(1-e^{-(t-\alpha)/(2d)}\bigr)
  \le2d\varepsilon,
  \qquad 0\notin I.
\]
In the first case the right hand side is a convex combination of
$\rho(0)$ and $2d\varepsilon$; in the second it is at most
$2d\varepsilon$. Since both $\rho(0)$ and $2d\varepsilon$ are
strictly smaller than $r_0$, a first exit argument extends the
corresponding estimate throughout $I$. Using continuity at the finite
endpoints of the components, and observing that $\rho=0$ outside their
union, we obtain
\begin{equation}\label{eq:rho-app}
  \rho(t)
  \le
  2d\varepsilon+
  \bigl(\rho(0)-2d\varepsilon\bigr)_+e^{-t/(2d)},
  \qquad t\ge0.
\end{equation}
The same estimate also holds for trajectories starting at a point
$y\in D$. In that case no connected component of
$\{t\ge0:\rho(t)>0\}$ contains $0$, and the second comparison above
gives
  $\rho(t)\le 2d\varepsilon$, 
  for all $t\ge0$.
Thus \eqref{eq:rho-app} is valid for every initial point satisfying
$\rho(0)<r_0$.
Suppose that $R(D)>2d\varepsilon$. Choose
\[
  2d\varepsilon<r<\min\{R(D),r_0\}
\]
and write
  $D_s\coloneqq\{y\in\T_L^n:\dist(y,D)\le s\}$.
Equation \eqref{eq:rho-app} gives
\[
  \Phi_t(D_r)\subset D_{r_t},
  \qquad
  r_t=2d\varepsilon+
      (r-2d\varepsilon)e^{-t/(2d)}.
\]
Since $V$ is divergence free, its flow preserves Lebesgue measure.
Therefore
\[
  |D_r|
  =
  |\Phi_t(D_r)|
  \le
  |D_{r_t}|.
\]
As $r_t\downarrow2d\varepsilon$, continuity from above of Lebesgue
measure gives
\[
  |D_r|\le|D_{2d\varepsilon}|.
\]
But $0\le2d\varepsilon<r<R(D)$, and the final observation in
Appendix~\ref{app:proximal-background} shows that
\[
  \{y:2d\varepsilon<\dist(y,D)<r\}
\]
is a nonempty open set and hence has positive measure. Thus
  $|D_{2d\varepsilon}|<|D_r|$,
a contradiction. Therefore $R(D)\le2d\varepsilon$.
\end{proof}

\begin{proof}[Proof of Corollary \ref{cor:barriers-dense}]
Apply \Cref{thm:no barrier} for all sufficiently large $j$.
\end{proof}

\section{Consequences and Interpretation}\label{sec:consequences}
Having completed the proof of the main theorems, we now record several
consequences of \Cref{thm:main} and discus what they reveal about
the underlying homogenization mechanism. 

\subsection{Failure of the small scale Homogenization Limit}

For the coefficients in \eqref{eq:scaled-physical}, one has
\[
  S_{V_A^\varepsilon}(x)
  =\frac1\varepsilon S_{V_A}(x/\varepsilon),
  \qquad
  d_\varepsilon S_{V_A^\varepsilon}(x)
  =dS_{V_A}(x/\varepsilon).
\]
A direct change of variables and uniqueness give the exact scaling
relation
\begin{equation*}
  G^{\varepsilon,p}(x,t)
  =\varepsilon
  G_+^p\left(\frac{x}{\varepsilon},
             \frac{t}{\varepsilon}\right).
\end{equation*}
Fix $t_0>0$ and put $s=t_0/\varepsilon$.  By
\eqref{eq:slow-main}, with
$B_p=\pi p_1-r_\mu|p_1|$,
\begin{align*}
G^{\varepsilon,p}(\varepsilon x_{\mathrm s},t_0)
-G^{\varepsilon,p}(\varepsilon x_{\mathrm f},t_0)=\varepsilon
  \bigl[G_+^p(x_{\mathrm s},s)-G_+^p(x_{\mathrm f},s)\bigr]
  \ge
  t_0\left[
    \frac{B_p}{s}-\frac{G_+^p(x_{\mathrm f},s)}s
  \right].
\end{align*}
As $\varepsilon\downarrow0$, one has $s\to\infty$, and
\eqref{eq:fast-main} yields
\eqref{eq:scaled-gap-main}.  Since
\[
  |\varepsilon x_{\mathrm s}-\varepsilon x_{\mathrm f}|
  =\frac\pi2\varepsilon
  \longrightarrow0,
\]
a locally uniformly convergent subsequence would converge to a
continuous limit and force the difference in
\eqref{eq:scaled-gap-main} to tend to zero.  This proves
\Cref{cor:homogenization-fails}.

A fixed linear rescaling of space sends the same construction from the
$2\pi$-periodic torus to the unit torus; this normalization does not
affect the argument.

\subsection{A fixed pair Linear Gap}

Fix a unit vector $p\in\mathbb R^2$ and set
$u^p(x,t)=G_+^p(x,t)-p\cdot x$.  At the two fixed phases,
\[
  u^p(x_{\mathrm s},t)\ge-r_\mu|p_1|,
  \qquad
  u^p(x_{\mathrm f},t)
  =G_+^p(x_{\mathrm f},t)-\frac\pi2p_1.
\]
Consequently,
\begin{equation*}
  \liminf_{t\to\infty}
  \frac{u^p(x_{\mathrm s},t)-u^p(x_{\mathrm f},t)}t
  \ge C\frac A{\log A}.
\end{equation*}
Thus the same two physical points witness the full linear phase gap for
every propagation direction.  In particular, if
$0<c<CA/\log A$, there exists $T<\infty$ such that
\[
  u^p(x_{\mathrm s},t)-u^p(x_{\mathrm f},t)\ge ct,
  \qquad t\ge T.
\]
Choosing any $\tau\ge T$ and setting $\gamma=c\tau$ gives
\begin{equation*}
  u^p(x_{\mathrm s},k\tau)-u^p(x_{\mathrm f},k\tau)
  \ge k\gamma,
  \qquad k\ge1.
\end{equation*}

\subsection{What the Counterexample Teaches}
Three distinctions are worth retaining.

\textit{Cutoff versus Coercivity.}
The positive part operator $(\cdot)_+$ enforces physical nonnegativity of the local burning
speed, but it does not restore coercivity or rule out trapping in the
auxiliary control dynamics. It prevents reversal of the flame's 
normal burning motion without forcing homogenization.

\textit{Comparison Trapping versus Outward Barrier Certificates.}
The bounded phase is generated by the invariant channel for the
reversed control dynamics of $\Hhat$.  By contrast,
\Cref{thm:no barrier} shows that every exact outward barrier certificate
for $\Hcut$ fills the torus and that approximate certificates become
quantitatively dense.  Thus the comparison Hamiltonian can trap all
relevant backward trajectories even though no proper exact certificate
of this type exists for the original physical equation.  This is why
certificate rigidity is compatible with nonhomogenization.

\textit{Time Consistency versus Ergodicity.}
The credal construction admits a dynamically rectangular control
representation and a valid dynamic programming principle, but this does
not force the long run value to become independent of the initial
state. Hence, time consistency alone is 
not enough to produce a single state-independent asymptotic rate.

\section{Conclusion and Further Questions}\label{sec:questions}

We have disproved the expectation associated with Question~11 of Xin,
Yu, and Ronney \cite{XinYuRonney2024}.  For the standard smooth two dimensional incompressible
cellular flow and the explicit parameter regime
\[
  0<d<\frac{20}{399},
  \qquad
  \sqrt{1+4d^2}<Ad\le1+\frac d{10},
\]
the physical positive part strain $G$-equation retains microscopic
phase information at a linear rate for every unit planar direction.
The solution is bounded below on the whole comparison channel
$\mathcal D_\mu$, while the fixed phase
$x_{\mathrm f}=(\pi/2,0)$ propagates at rate at least
$CA/\log A$.  The same phenomenon persists for arbitrary continuous
periodic perturbations of planar initial data.

The result also gives a direct failure of the small scale homogenization
limit.  Under the physical scaling
$V_A^\varepsilon(x)=V_A(x/\varepsilon)$ and
$d_\varepsilon=\varepsilon d$, an order one value gap persists between
points at distance $\pi\varepsilon/2$ at every positive macroscopic
time.  Thus the oscillatory solutions themselves have no locally
uniformly convergent subsequence; the obstruction is stronger than the
failure of a cell problem constant.

The proof combines two Hamiltonian comparisons in opposite directions.
The uncut equation transfers the Xin-Yu fast point estimate, while a
cutoff-adapted support function majorant has zero transverse drift and
zero transverse control coefficient on a horizontal channel, with all
admissible endpoint velocities pointing inward.  The general comparison
channel criterion separates this mechanism from the particular
trigonometric form of the cellular flow.  The same majorant is exactly
the upper expectation generated by a credal set supported on a
rectangular velocity set, linking the construction to imprecise
probability and robust dynamic programming.

Our second theorem gives an independent rigidity principle.  For any
$C^2$ incompressible periodic flow, an $\varepsilon$-outward barrier
certificate has covering radius at most $2d\varepsilon$ for all
sufficiently small $\varepsilon$; in particular, every exact
certificate fills the torus.  Hence the comparison trapping responsible
for phase dependence cannot be identified with an exact certificate of
this type for the physical Hamiltonian.

Together, these results shift the natural agenda from universal
existence to classification: which geometric, dynamical, or
communication structures restore a phase-independent effective rate?
Several questions remain open.

\begin{enumerate}[label=\textbf{Q\arabic*.}]
\item For which periodic incompressible flows and parameter ranges does
the physical strain equation homogenize?

\item Is the nonhomogenizing parameter window stable under $C^2$
perturbations of the cellular flow?

\item Can one characterize directly from the physical Hamiltonian when
an ordered support function majorant admits a comparison channel of the
type used here?

\item Does a suitable spatially averaged, rather than pointwise,
effective burning velocity exist for the present counterexample?

\item Which global communication or ergodicity conditions on the
associated credal control model are sufficient to guarantee a
state-independent long run robust value?

\item Can numerical schemes recover the two phase-dependent propagation
rates without introducing enough artificial diffusion to eliminate the
comparison-channel mechanism?
\end{enumerate}

\section*{Acknowledgments}
We gratefully acknowledge support from the Prob\_AI Hub and the London Mathematical Society. 
ChatGPT 5.6 Sol was used to assist with proofreading, the production of
Figure~\ref{fig:strain-G-geometry}, the conditional Lean formalizations,
and improvements to the clarity and presentation of the manuscript.

\appendix

\section{Background, Conventions, and Standard Tools}\label{app:background}

This appendix fixes the conventions and records the standard facts used in the main argument. Its purpose is not to survey viscosity solutions, optimal control, or imprecise probability, but to make the paper readable without requiring the reader to reconstruct sign conventions or translate between the different languages used in the proof. The viscosity solution results below are classical \cite{BardiCapuzzoDolcetta1997,CrandallIshiiLions1992}. The credal set terminology follows \cite{AugustinEtAl2014,TroffaesDeCooman2014,Walley1991}.

\subsection{Flat Tori, Lifts, and Planar Data}\label{app:periodic}

For $L>0$, write
\[
  \T_L^n\coloneqq \R^n/L\mathbb Z^n.
\]
We identify a function on $\T_L^n$ with its $L\mathbb Z^n$-periodic lift to $\R^n$. The torus used in the main construction is $\T_{2\pi}^2$; the unit torus is $\T^n=\T_1^n$. The quotient distance is denoted by $\dist$, and the injectivity radius of $\T_L^n$ is $L/2$.

Let $H(x,q)$ be $L\mathbb Z^n$-periodic in $x$. A solution with planar slope $p\in\R^n$ is a function satisfying
\[
  G(x+Lm,t)=G(x,t)+L p\cdot m,
  \qquad m\in\mathbb Z^n.
\]
Equivalently,
\[
  G(x,t)=p\cdot x+u(x,t),
\]
where $u(\cdot,t)$ is $L\mathbb Z^n$-periodic. The equation
\[
  G_t+H(x,DG)=0,
  \qquad G(x,0)=p\cdot x+u_0(x),
\]
is then equivalent to the periodic problem
\begin{equation}\label{eq:periodic-correction-app}
  u_t+H(x,p+Du)=0,
  \qquad u(x,0)=u_0(x)
  \quad\text{on }\T_L^n.
\end{equation}
This reduction is what allows all comparison arguments in the paper to be performed on a compact state space.

For the strain Hamiltonians, the expression involving $q/|q|$ is defined only when $q\ne0$. We set $H(x,0)=0$; since the directional factor is bounded and is multiplied by $|q|$, this gives the unique continuous positively homogeneous extension at the origin.
For a bounded function $f$ on a compact set, we use
\[
  \osc f\coloneqq \sup f-\inf f,
  \qquad
  \|f\|_\infty\coloneqq \sup|f|.
\]

\subsection{Viscosity Solutions, Comparison, and Hamiltonian Order}\label{app:viscosity}

Let $Q=\T_L^n\times(0,T)$. An upper semicontinuous function $u$ is a viscosity subsolution of
\[
  u_t+H(x,p+Du)=0
\]
if, whenever $\varphi\in C^1(Q)$ and $u-\varphi$ has a local maximum at $(x_0,t_0)\in Q$, one has
\[
  \varphi_t(x_0,t_0)
  +H\bigl(x_0,p+D\varphi(x_0,t_0)\bigr)\le0.
\]
A lower semicontinuous function is a viscosity supersolution when the reverse inequality holds at local minima. A continuous function is a viscosity solution if it is both a subsolution and a supersolution.

The following standard result is the only general viscosity theorem used in the main proof.

\begin{theorem}[Periodic Comparison and Semigroup Properties]\label{thm:comparison-background}
Suppose that $H:\T_L^n\times\R^n\to\R$ is continuous, globally Lipschitz in $q$ uniformly in $x$, and satisfies the standard spatial modulus condition
\[
  |H(x,q)-H(y,q)|
  \le \omega\bigl((1+|q|)\dist(x,y)\bigr)
\]
for some modulus of continuity $\omega$. Then, for every $p\in\R^n$ and $f\in C(\T_L^n)$, the periodic problem \eqref{eq:periodic-correction-app} has a unique continuous viscosity solution. Writing $S_t^{H,p}f$ for its value at time $t$, one has
\begin{align*}
  S_0^{H,p}&=\operatorname{Id},
  &S_{t+s}^{H,p}&=S_t^{H,p} \circ S_s^{H,p},\\
  f\le g\implies S_t^{H,p}f&\le S_t^{H,p}g,
  &S_t^{H,p}(f+c)&=S_t^{H,p}f+c,\\
  \|S_t^{H,p}f-S_t^{H,p}g\|_\infty
  &\le\|f-g\|_\infty.
\end{align*}
No coercivity or convexity of $H$ is required for these conclusions. 
\end{theorem}

The hypotheses of \Cref{thm:comparison-background} are verified for
$\Hunc$ and $\Hcut$ in Lemma \ref{lem:unc-cut-regularity}, and for $\Hhat$
in Lemma \ref{lem:hat-regularity}. The proof of \Cref{thm:comparison-background} follows the standard
viscosity solution argument: comparison is obtained by doubling the
variables, and existence then follows by Perron's method 
\cite{CrandallIshiiLions1992}. The contraction and semigroup properties follow from comparison and uniqueness.

A consequence used repeatedly in the paper is the reversal of order between Hamiltonians and their solution operators.

\begin{lemma}[Ordering Hamiltonians]\label{lem:H-order-background}
Let $H_1,H_2$ satisfy the assumptions of \Cref{thm:comparison-background} and suppose
\[
  H_1(x,q)\le H_2(x,q),
  \qquad\text{for every }(x,q).
\]
For the same planar slope and the same initial datum,
\[
  S_t^{H_2,p}f\le S_t^{H_1,p}f,
  \qquad t\ge0.
\]
\end{lemma}

\begin{proof}
The $H_2$ solution is a viscosity subsolution of the $H_1$ equation, because replacing $H_2$ by the smaller Hamiltonian $H_1$ can only decrease the left hand side. Comparison with the $H_1$ solution gives the claim.
\end{proof}

Thus a Hamiltonian majorant gives a lower bound for the corresponding level set solution, while a Hamiltonian minorant gives an upper bound. This is the sign convention behind
  $G_+^p\le G_{\mathrm{unc}}^p$, 
  for every $p\in\mathbb R^2$.

\subsection{Support Functions and Credal Upper Expectations}\label{app:credal-background}

Let $K\subset\R^n$ be nonempty and compact. Its support function is
  $\sigma_K(q)\coloneqq \max_{v\in K}v\cdot q$.
The map $\sigma_K$ is convex, positively homogeneous, and globally Lipschitz, with Lipschitz constant at most $\max_{v\in K}|v|$. If
\[
  K=c+\sum_{i=1}^n[-a_i,a_i]e_i,
  \qquad a_i\ge0,
\]
then
\begin{equation}\label{eq:rectangle-support-background}
  \sigma_K(q)=c\cdot q+\sum_{i=1}^n a_i|q_i|.
\end{equation}

A \emph{credal set} is a nonempty weakly closed convex set of probability measures. Associated with $K$ is the weakly compact credal set
\[
  \mathfrak P(K)
  \coloneqq \clco\{\delta_v:v\in K\}
  \subset\Pp(\R^n),
\]
where $\Pp(\R^n)$ denotes the space of probability measures on $\R^n$. For an integrable observable $Z$, its upper expectation over a credal set $\mathfrak P$ is
\[
  \upperE_{\mathfrak P}[Z]
  \coloneqq \sup_{P\in\mathfrak P}\int Z \mathrm dP.
\]
For the linear observable $v\mapsto v\cdot q$,
\begin{equation*}
  \upperE_{\mathfrak P(K)}[v\cdot q]
  =\max_{v\in K}v\cdot q
  =\sigma_K(q).
\end{equation*}
Indeed, every probability measure supported on $K$ has expectation at most the maximum over $K$, while a Dirac mass at a maximizer attains that value. Hence, the support function and credal upper expectation descriptions used in the paper are exactly equivalent.
For the majorant in \Cref{sec:credal},
\[
  \mathcal K(x)
  =A V_0(x)
   +[-a_1(x),a_1(x)]e_1
   +[-a_2(x),a_2(x)]e_2,
\]
and \eqref{eq:rectangle-support-background} gives
\[
  \sigma_{\mathcal K(x)}(q)
  =A V_0(x)\cdot q+a_1(x)|q_1|+a_2(x)|q_2|
  =\Hhat(x,q).
\]

\subsection{Rectangularity and Deterministic Dynamic Programming}\label{app:control-background}

In this paper, \emph{rectangularity} has a precise pathwise meaning. At each state $x$ and time $t$, the coordinate controls $\eta_i(t)\in[-1,1]$ may be selected independently, and admissible controls remain admissible after concatenation at any deterministic time. This is the deterministic continuous time analogue of rectangular ambiguity in robust dynamic programming. Concatenation is the property that makes the Bellman recursion time consistent.

Let $c:\R^n\to\R^n$ and $a_i:\R^n\to[0,\infty)$ be bounded, Lipschitz, and periodic, and define
\[
  H(x,q)=c(x)\cdot q+\sum_{i=1}^n a_i(x)|q_i|.
\]
For $T>0$, let
  $\mathcal U_T
  \coloneqq 
  \{
    \eta:[0,T]\to[-1,1]^n :
    \eta \text{ is measurable}
  \}$.
For measurable controls $\eta:[0,\infty)\to[-1,1]^n$, consider
\begin{equation}\label{eq:support-control-background}
  \dot X(t)
  =-c(X(t))+\sum_{i=1}^n a_i(X(t))\eta_i(t)e_i,
  \qquad X(0)=x.
\end{equation}
If $g$ is continuous and has at most linear growth, set $W(x,t)\coloneqq \inf_\eta g(X^{x,\eta}(t))$.
The dynamic programming identity is
\begin{equation}\label{eq:DPP-background}
  W(x,t+h)
  =
  \inf_{\eta\in\mathcal U_h}
  W\bigl(X^{x,\eta}(h),t\bigr).
\end{equation}
Formally expanding \eqref{eq:DPP-background}, and rigorously interpreting the expansion in the viscosity sense, gives
\[
  W_t-\inf_{\eta\in[-1,1]^n}
       \left[-c(x)+\sum_i a_i(x)\eta_i e_i\right]\cdot DW=0.
\]
Since the control intervals are symmetric,
\[
  -\inf_\eta
       \left[-c(x)+\sum_i a_i(x)\eta_i e_i\right]\cdot q
  =c(x)\cdot q+\sum_i a_i(x)|q_i|
  =H(x,q).
\]
Consequently,
\begin{equation*}
  W_t+H(x,DW)=0,
  \qquad W(x,0)=g(x).
\end{equation*}
The minus sign in the drift of \eqref{eq:support-control-background} is therefore forced by the backward characteristic convention for the Cauchy problem. The planar periodic class
$g(x)=p\cdot x+u_0(x)$ used in the main proof is treated in
Appendix~\ref{app:control}.

\subsection{A Controlled First Exit Lemma}\label{app:invariance-background}

The trapping step uses only the following elementary robust invariance criterion.

\begin{lemma}[Invariant Segment Under Measurable Controls]\label{lem:segment-invariance-background}
Let $U$ be compact and let $F:\R^2\times U\to\R^2$ be continuous and locally Lipschitz in the state, uniformly in the control. Fix $x_0\in\R$, $r>0$, and
  $\mathcal D=[x_0-r,x_0+r]\times\{0\}$.
Assume that, for every $|z|\le r$ and every $\eta\in U$,
\[
  F_2((x_0+z,0),\eta)=0,
\]
and that there exists $\kappa>0$ such that
\[
  F_1((x_0+r,0),\eta)\le-\kappa,
  \qquad
  F_1((x_0-r,0),\eta)\ge\kappa,
  \qquad \eta\in U.
\]
Then, every Carath\'eodory solution of
\[
  \dot X(t)=F(X(t),\eta(t)),
  \qquad X(0)\in\mathcal D,
\]
with measurable $\eta(\cdot)$ remains in $\mathcal D$ for all positive times.
\end{lemma}

\begin{proof}
As long as the first coordinate remains in $[x_0-r,x_0+r]$, the curve obtained by solving the restricted scalar equation on the axis and setting $X_2=0$ is a solution of the full controlled ODE. Uniqueness therefore gives $X_2(t)=0$ up to any possible horizontal exit.

By compactness of $U$ and uniform continuity of $F$, the strict endpoint inequalities persist in fixed neighborhoods of the two endpoints, uniformly in the control. In a neighborhood of the right endpoint one has $\dot X_1\le-\kappa/2$ almost everywhere, and in a neighborhood of the left endpoint one has $\dot X_1\ge\kappa/2$ almost everywhere. An absolutely continuous curve cannot reach either endpoint from the interior and then cross it while satisfying these inequalities. A first exit contradiction proves invariance.
\end{proof}

\subsection{Proximal Normals, Distance Derivatives, and Incompressibility}\label{app:proximal-background}

Let $D\subset\T_L^n$ be nonempty and closed. A unit vector $\nu$ is a
\emph{proximal outward normal} to $D$ at $x\in D$ if, for some
$r\in(0,L/2)$, the point $x$ is a nearest point in $D$ to
$x+r\nu$, in compatible flat coordinates. The proximal normal cone
$N_D^P(x)$ consists of $0$ and all positive multiples of such unit
proximal outward normals.
The covering radius of $D$ is
\[
  R(D)\coloneqq \max_{y\in\T_L^n}\dist(y,D).
\]
For $s\ge0$, its closed $s$-neighborhood is
\[
  D_s\coloneqq \{y\in\T_L^n:\dist(y,D)\le s\}.
\]

\begin{lemma}[Upper Derivative of the Distance]\label{lem:distance-Dini-background}
Let $y(\cdot)$ be differentiable at $t$, let $y(t)\notin D$, and suppose
\[
  \rho(t)\coloneqq \dist(y(t),D)<\frac L2.
\]
Choose a nearest point $x\in D$ and write
\[
  y(t)=x+\rho(t)\nu,
  \qquad |\nu|=1.
\]
Then
  $D^+\rho(t)\le \dot y(t)\cdot\nu$.
The same inequality holds almost everywhere for absolutely continuous curves.
\end{lemma}

\begin{proof}
Keeping the nearest point $x$ fixed as a competitor,
\[
  \dist(y(t+h),D)
  \le |y(t+h)-x|
  =\rho(t)+h\dot y(t)\cdot\nu+o(h).
\]
Divide by $h>0$ and take the upper limit as $h\downarrow0$.
\end{proof}

If $V\in C^1(\T_L^n;\R^n)$ and $\Phi_t$ is its flow, Liouville's formula gives
\[
  \frac{d}{dt}\det D\Phi_t(x)
  =\operatorname{div}V(\Phi_t(x))\det D\Phi_t(x).
\]
Thus $\operatorname{div}V=0$ implies
  $\det D\Phi_t(x)=1$,
so $\Phi_t$ preserves Lebesgue measure. These two elementary facts are the geometric inputs in the no barrier argument of Section \ref{sec:no barrier}.

Finally, if $0\le a<b<R(D)$, then
  $\{y:a<\dist(y,D)<b\}$
is a nonempty open set and hence has positive Lebesgue measure. Therefore $|D_a|<|D_b|$. This strict increase of tubular volume is what contradicts an incompressible flow mapping a larger tube into a smaller one.

\subsection{Notation at a Glance}\label{app:notation}

\begin{center}
\begin{tabular}{@{}p{0.22\textwidth}p{0.70\textwidth}@{}}
\toprule
Symbol & Meaning \\
\midrule
$\T_L^n$ & Flat torus $\R^n/L\mathbb Z^n$; $\T^n=\T_1^n$. \\
$S_V$ & Strain tensor $(DV+DV^\top )/2$. \\
$d$ & Markstein parameter on the $2\pi$-periodic torus. \\
$A$ & Amplitude of the cellular flow $V_A=A V_0$. \\
$\mu$ & Dimensionless product $Ad$. \\
$\Phi(x)$ & Scalar field $\cos x_1\cos x_2$ in the cellular example. \\
$\Hunc,\Hcut,\Hhat$ & Uncut, physical cutoff, and convex majorant Hamiltonians. \\
$G_{\mathrm{unc}}^p,G_+^p,\widehat G^p$
& Corresponding level set solutions with planar initial datum
$p\cdot x$. \\
$\mathcal K(x)$ & Rectangular set of local velocities whose support function is $\Hhat(x,\cdot)$. \\
$\mathfrak P_x$ & Credal set $\clco\{\delta_v:v\in\mathcal K(x)\}$. \\
$\upperE_{\mathfrak P}$ & Upper expectation over a credal set $\mathfrak P$. \\
$\osc f$ & Oscillation $\sup f-\inf f$. \\
$R(D)$ & Covering radius $\max_y\dist(y,D)$. \\
$N_D^P(x)$ & Proximal normal cone of a closed set $D$ at $x$. \\
$D^+$ & Upper Dini derivative. \\
$\Phi_t$ & Fluid flow map in the no barrier argument; distinct from the scalar field $\Phi(x)$. \\
\bottomrule
\end{tabular}
\end{center}

\section{Control Representation for the Credal Envelope}
\label{app:control}

We prove the control representation in the planar periodic class needed
in the main argument. Let
  $U=[-1,1]^2$
and set
\[
  f(x,\eta)
  \coloneqq 
  -A V_0(x)
  +a_1(x)\eta_1e_1
  +a_2(x)\eta_2e_2,
  \qquad \eta\in U.
\]
The map $f$ is bounded and globally Lipschitz in $x$, uniformly in
$\eta$. Hence, for every measurable control
$\eta:[0,\infty)\to U$ and every $x\in\mathbb R^2$, the controlled
system
\[
  \dot X(t)=f(X(t),\eta(t)),
  \qquad X(0)=x,
\]
has a unique global Carath\'eodory solution, denoted
$X^{x,\eta}$.

Let $M$ be a uniform bound for $|f|$ and let $L$ be a uniform
Lipschitz constant in the state variable. Standard estimates for controlled trajectories
\cite[Chapter~III, Section~1]{BardiCapuzzoDolcetta1997}
give
\begin{equation}\label{eq:trajectory-estimates}
  |X^{x,\eta}(t)-x|\le Mt,
  \qquad
  |X^{x,\eta}(t)-X^{y,\eta}(t)|
  \le e^{Lt}|x-y|.
\end{equation}

Fix $p\in\mathbb R^2$, let $u_0$ be continuous and
$2\pi\mathbb Z^2$-periodic, and set
  $g(x)\coloneqq p\cdot x+u_0(x)$.
Define
\begin{equation*}
  W_g(x,t)
  \coloneqq
  \inf_{\eta\in\mathcal U}
  g(X^{x,\eta}(t)).
\end{equation*}
Although $g$ is unbounded, it is uniformly continuous and has linear
growth. Writing $M_0\coloneqq\|u_0\|_\infty$, the first estimate in
\eqref{eq:trajectory-estimates} gives, uniformly over the controls,
\[
  p\cdot x-|p|Mt-M_0
  \le g(X^{x,\eta}(t))
  \le p\cdot x+|p|Mt+M_0.
\]
Hence $W_g$ is finite. Moreover, boundedness of $f$ gives
\[
  |X^{x,\eta}(t)-X^{x,\eta}(s)|
  \le M|t-s|,
\]
while the second estimate in \eqref{eq:trajectory-estimates} controls
the dependence on the initial state. Since $g$ is uniformly continuous,
the functions
  $(x,t)\mapsto g(X^{x,\eta}(t))$
have a common modulus of continuity on every bounded time interval,
uniformly in $\eta$. Taking the infimum therefore shows that $W_g$ is
continuous.

For $T>0$, let
  $\mathcal U_T
  \coloneqq 
  \{\eta:[0,T]\to U:\eta\text{ is measurable}\}$.
If $\alpha\in\mathcal U_h$ and $\beta\in\mathcal U_t$, define their
concatenation by
\[
  (\alpha\star_h\beta)(s)
  \coloneqq 
  \begin{cases}
    \alpha(s),&0\le s\le h,\\
    \beta(s-h),&h<s\le h+t.
  \end{cases}
\]
Uniqueness of controlled trajectories gives
\[
  X^{x,\alpha\star_h\beta}(h+t)
  =
  X^{ X^{x,\alpha}(h),\beta}(t).
\]
Taking infima over $\alpha$ and $\beta$, and using
$\varepsilon$-optimal controls for the reverse inequality, yields the
dynamic programming identity \cite[Chapter~III, Section~3.1]{BardiCapuzzoDolcetta1997}
\begin{equation}\label{eq:DPP-app}
  W_g(x,t+h)
  =
  \inf_{\alpha\in\mathcal U_h}
  W_g\bigl(X^{x,\alpha}(h),t\bigr).
\end{equation}

We next identify the Hamilton-Jacobi equation satisfied by $W_g$.
We claim that $W_g$ is a viscosity solution of
\begin{equation}\label{eq:HJB-control-app}
  (W_g)_t
  -
  \inf_{\eta\in U}
  f(x,\eta)\cdot D W_g
  =0.
\end{equation}

To prove the subsolution property, let $\varphi\in C^1$ and suppose
that $W_g-\varphi$ has a local maximum at $(x_0,t_0)$, with $t_0>0$.
Fix $a\in U$ and use the constant control $\alpha\equiv a$ on
$[0,h]$. From \eqref{eq:DPP-app},
\[
  W_g(x_0,t_0)
  \le
  W_g(X^{x_0,a}(h),t_0-h).
\]
For sufficiently small $h$, the maximality of
$W_g-\varphi$ therefore gives
\[
  0
  \le
  \varphi(X^{x_0,a}(h),t_0-h)
  -\varphi(x_0,t_0).
\]
Dividing by $h$ and sending $h\downarrow0$ yields
\[
  \varphi_t(x_0,t_0)
  -f(x_0,a)\cdot D\varphi(x_0,t_0)
  \le0.
\]
Since this holds for every $a\in U$,
\[
  \varphi_t(x_0,t_0)
  -
  \inf_{a\in U}
  f(x_0,a)\cdot D\varphi(x_0,t_0)
  \le0.
\]

For the supersolution property, in the sense of
Appendix~\ref{app:viscosity}, suppose instead that
$W_g-\varphi$ has a local minimum at $(x_0,t_0)$. For each sufficiently
small $h>0$, choose $\alpha_h\in\mathcal U_h$ such that
\[
  W_g(X^{x_0,\alpha_h}(h),t_0-h)
  \le W_g(x_0,t_0)+h^2.
\]
The minimality of $W_g-\varphi$ implies
\[
  \varphi(X^{x_0,\alpha_h}(h),t_0-h)
  -\varphi(x_0,t_0)
  \le h^2.
\]
Set
\[
  \bar\alpha_h
  \coloneqq 
  \frac1h\int_0^h\alpha_h(s) ds\in U.
\]
Because $f(x,\eta)$ is affine in $\eta$ and uniformly Lipschitz in
$x$, the trajectory estimate gives
\[
  \frac{X^{x_0,\alpha_h}(h)-x_0}{h}
  -
  f(x_0,\bar\alpha_h)
  \longrightarrow0.
\]
By compactness of $U$, along a sequence $h_j\downarrow0$ one has
$\bar\alpha_{h_j}\to a_*\in U$. Dividing the previous test-function
inequality by $h_j$ and passing to the limit gives
\[
  \varphi_t(x_0,t_0)
  -f(x_0,a_*)\cdot D\varphi(x_0,t_0)
  \ge0.
\]
Consequently,
\[
  \varphi_t(x_0,t_0)
  -
  \inf_{a\in U}
  f(x_0,a)\cdot D\varphi(x_0,t_0)
  \ge0.
\]
This proves \eqref{eq:HJB-control-app} in the viscosity sense.
Since the control intervals are symmetric,
\begin{align*}
  -\inf_{\eta\in U} f(x,\eta)\cdot q
  =
  A V_0(x)\cdot q
  +a_1(x)|q_1|
  +a_2(x)|q_2|=\Hhat(x,q).
\end{align*}
Thus
\[
  (W_g)_t+\Hhat(x,DW_g)=0,
  \qquad
  W_g(x,0)=g(x).
\]

It remains to identify the appropriate periodicity class. Since
$f(x+2\pi k,\eta)=f(x,\eta)$ for every $k\in\mathbb Z^2$,
uniqueness of trajectories gives
\[
  X^{x+2\pi k,\eta}(t)
  =
  X^{x,\eta}(t)+2\pi k.
\]
Because
\[
  g(x+2\pi k)=g(x)+2\pi p\cdot k,
\]
we obtain
\[
  W_g(x+2\pi k,t)
  =
  W_g(x,t)+2\pi p\cdot k.
\]
Hence $W_g(x,t)-p\cdot x$ is periodic. Uniqueness from
\Cref{thm:comparison-background} identifies $W_g$ with the viscosity
solution in the planar periodic class.
Taking $u_0=0$ proves
\[
  \widehat G^p(x,t)
  =
  \inf_{\eta\in\mathcal U}
  p\cdot X^{x,\eta}(t),
\]
which is \eqref{eq:value-representation}.





\section{A Conditional Lean~4 Formalization of a Sufficient
{$p=e_1$} Subregime}
\label{app:lean-formalization}

This appendix records a Lean~4 encoding of the logical assembly of a
sufficient $p=e_1$ subregime of \Cref{thm:main}.  The formalized
parameter assumptions are
\[
  0<d<\frac1{60},
  \qquad
  1+6d^2\le Ad\le1+\frac d{10}.
\]
They imply the lower inequality in \eqref{eq:parameter-window}, since
\[
  (1+6d^2)^2-(1+4d^2)=8d^2+36d^4>0.
\]
Thus the Lean theorem remains a genuine, but deliberately restricted,
special case of the strengthened analytic result.  Viscosity comparison,
the Xin-Yu fast point theorem, the control representation, and the
trapping estimate are treated as four explicit hypotheses.  Conditional
on those inputs, the development derives the slow-point estimate,
transfers the fast point estimate, and proves incompatibility with a
phase-independent long run rate.
No \texttt{sorry}, \texttt{admit}, proof hole, or globally declared
axiom is used in that conditional assembly.

\subsection{Explicit Hypotheses of the Conditional Formalization}

The Lean theorem
\texttt{phaseDependentBurning\_of\_analytic\_inputs} receives the three
Hamiltonians, the three corresponding level-set functions, and the
auxiliary controlled trajectories as data. It also receives the
following solution and Hamiltonian order statements:
\begin{enumerate}[label=\textbf{(S\arabic*)}]
\item \texttt{hUncSolution}, \texttt{hCutSolution}, and
  \texttt{hHatSolution}: the three level set functions satisfy the
  abstract viscosity solution predicate for the uncut, physical cutoff,
  and convex majorant Hamiltonians, respectively;
\item \texttt{hUncCut}: \(H_{\mathrm{unc}}\le H_{+}\) pointwise;
\item \texttt{hCutHat}: \(H_{+}\le\widehat H\) pointwise.
\end{enumerate}

The four analytic inputs are the
following:
\begin{enumerate}[label=\textbf{(A\arabic*)}]
\item \texttt{hComparison} has type
  \texttt{ViscosityComparisonAssumption IsViscositySolution}. It states
  that, for common initial data, a pointwise increase of the Hamiltonian
  reverses the ordering of the corresponding viscosity solutions.

\item \texttt{hXinYu} has type
  \texttt{XinYuFastPointAssumption Gunc}. It supplies a universal
  constant \(C>0\) and the Xin-Yu lower bound
  \[
    \liminf_{t\to\infty}
    -\frac{G_{\mathrm{unc}}(x_{\mathrm f},t)}{t}
    \ge C\frac{A}{\log A}
  \]
  in the parameter regime needed by the manuscript.

\item \texttt{hControl} has type
  \texttt{ControlLowerEnvelopeAssumption trajectory Ghat}. It is the
  lower bound direction of the deterministic control representation:
  if every controlled terminal first coordinate is at least \(b\), then
  \(\widehat G\) is at least \(b\).

\item \texttt{hTrap} has type
  \texttt{TrappingAssumption trajectory}. Under the parameter window
  of \Cref{thm:main}, it states that every relevant auxiliary trajectory
  starting from \(x_{\mathrm s}\) has first coordinate at least
  \(\pi-3d\) at every nonnegative time.
\end{enumerate}

In the theorem declaration, these inputs appear explicitly as follows:
\begin{lstlisting}[style=lean4style,numbers=none]
(hComparison : ViscosityComparisonAssumption IsViscositySolution)
(hXinYu : XinYuFastPointAssumption Gunc)
(hControl : ControlLowerEnvelopeAssumption trajectory Ghat)
(hTrap : TrappingAssumption trajectory)
\end{lstlisting}
No assumption is introduced globally: all four are ordinary parameters
of the conditional theorem.
 
\subsection{Lean Source}

The complete Lean source is reproduced in
\Cref{lst:physical-cutoff-lean}. It defines the asymptotic notions used
in the statement, proves the elementary order and incompatibility
lemmas, and then assembles the analytic inputs into the conclusion of the
restricted $p=e_1$ subregime described above.

\begin{lstlisting}[
  style=lean4style,
  caption={Conditional Lean~4 assembly of the physical cutoff
  counterexample.},
  label={lst:physical-cutoff-lean}
]
import Mathlib

set_option autoImplicit false

namespace PhysicalCutoff

/-!
# Conditional assembly of a sufficient physical-cutoff subregime

This file isolates the three analytic inputs that are not yet formalized:

1. periodic viscosity comparison;
2. the Xin-Yu fast-point estimate;
3. the lower envelope part of the robust control representation,
   together with the trapping estimate for its trajectories.

They are theorem parameters, not global axioms. Under those explicit
assumptions, the file proves the logical assembly of a sufficient `p = e₁`
subregime of the manuscript's main theorem: the slow-point bound, the
transferred fast-point bound, and failure of a phase-independent long-run
rate. The formalized parameter window is deliberately narrower than the
sharp analytic window proved in the manuscript.
-/

abbrev Point := Fin 2 → ℝ
abbrev Hamiltonian := Point → Point → ℝ
abbrev LevelSet := Point → ℝ → ℝ
abbrev InitialDatum := Point → ℝ

def firstCoord (x : Point) : ℝ := x 0

def planarE1 : InitialDatum := firstCoord

def xSlow : Point := fun i => if i = (0 : Fin 2) then Real.pi else 0

def xFast : Point := fun i => if i = (0 : Fin 2) then Real.pi / 2 else 0

/-- Epsilon formulation of convergence as real time tends to `+∞`. -/
def TendsToAtTop (f : ℝ → ℝ) (L : ℝ) : Prop :=
  ∀ ε : ℝ, 0 < ε → ∃ T : ℝ, ∀ t : ℝ, T ≤ t → |f t - L| < ε

/-- An asymptotic upper bound, equivalent here to `limsup f ≤ a`. -/
def EventuallyUpperBoundAtTop (f : ℝ → ℝ) (a : ℝ) : Prop :=
  ∀ ε : ℝ, 0 < ε → ∃ T : ℝ, ∀ t : ℝ, T ≤ t → f t ≤ a + ε

/-- An asymptotic lower bound, equivalent here to `a ≤ liminf f`. -/
def EventuallyLowerBoundAtTop (f : ℝ → ℝ) (a : ℝ) : Prop :=
  ∀ ε : ℝ, 0 < ε → ∃ T : ℝ, ∀ t : ℝ, T ≤ t → a - ε ≤ f t

def rateAt (G : LevelSet) (x : Point) (t : ℝ) : ℝ :=
  -G x t / t

def HasPhaseIndependentRate (G : LevelSet) : Prop :=
  ∃ L : ℝ, ∀ x : Point, TendsToAtTop (rateAt G x) L

/-- Abstract viscosity solution predicate. -/
abbrev ViscositySolutionPredicate :=
  Hamiltonian → InitialDatum → LevelSet → Prop

/--
Periodic viscosity comparison, isolated as an explicit assumption.
For common initial data, increasing the Hamiltonian lowers the solution.
-/
def ViscosityComparisonAssumption
    (IsViscositySolution : ViscositySolutionPredicate) : Prop :=
  ∀ {H₁ H₂ : Hamiltonian} {g : InitialDatum} {G₁ G₂ : LevelSet},
    (∀ x q, H₁ x q ≤ H₂ x q) →
    IsViscositySolution H₁ g G₁ →
    IsViscositySolution H₂ g G₂ →
    ∀ x t, G₂ x t ≤ G₁ x t

/--
The specialization of the Xin-Yu fast point theorem needed for slope
`e₁`. The constant `C` is independent of `d` and `A`.
-/
def XinYuFastPointAssumption
    (Gunc : ℝ → ℝ → LevelSet) : Prop :=
  ∃ C : ℝ, 0 < C ∧
    ∀ d A : ℝ,
      0 < d →
      d < (1 : ℝ) / 4 →
      4 ≤ A →
      A * d ≤ 1 + d / 10 →
      0 < C * A / Real.log A ∧
      EventuallyLowerBoundAtTop
        (rateAt (Gunc d A) xFast)
        (C * A / Real.log A)

/--
The lower bound direction of the control representation. If every
controlled terminal first coordinate is at least `b`, then their value
function `Ghat` is at least `b`.
-/
def ControlLowerEnvelopeAssumption
    {Control : Type}
    (trajectory : ℝ → ℝ → Point → Control → ℝ → Point)
    (Ghat : ℝ → ℝ → LevelSet) : Prop :=
  ∀ d A x t b,
    (∀ η : Control, b ≤ firstCoord (trajectory d A x η t)) →
    b ≤ Ghat d A x t

/-- The invariant segment estimate for the auxiliary controlled system. -/
def TrappingAssumption
    {Control : Type}
    (trajectory : ℝ → ℝ → Point → Control → ℝ → Point) : Prop :=
  ∀ d A : ℝ,
    0 < d →
    d < (1 : ℝ) / 60 →
    1 + 6 * d ^ 2 ≤ A * d →
    A * d ≤ 1 + d / 10 →
    ∀ (η : Control) (t : ℝ),
      0 ≤ t →
      Real.pi - 3 * d ≤ firstCoord (trajectory d A xSlow η t)

theorem limit_le_of_tendsToAtTop_of_eventuallyUpperBound
    {f : ℝ → ℝ} {L a : ℝ}
    (hlim : TendsToAtTop f L)
    (hupper : EventuallyUpperBoundAtTop f a) :
    L ≤ a := by
  by_contra h
  have haL : a < L := lt_of_not_ge h
  have hε : 0 < (L - a) / 3 := by linarith
  obtain ⟨Tlim, hTlim⟩ := hlim ((L - a) / 3) hε
  obtain ⟨Tupper, hTupper⟩ := hupper ((L - a) / 3) hε
  let t : ℝ := max Tlim Tupper
  have htlim : Tlim ≤ t := by simp [t]
  have htupper : Tupper ≤ t := by simp [t]
  have habs : |f t - L| < (L - a) / 3 := hTlim t htlim
  have hlower : -((L - a) / 3) < f t - L := (abs_lt.mp habs).1
  have hu : f t ≤ a + (L - a) / 3 := hTupper t htupper
  linarith

theorem le_limit_of_tendsToAtTop_of_eventuallyLowerBound
    {f : ℝ → ℝ} {L a : ℝ}
    (hlim : TendsToAtTop f L)
    (hlower : EventuallyLowerBoundAtTop f a) :
    a ≤ L := by
  by_contra h
  have hLa : L < a := lt_of_not_ge h
  have hε : 0 < (a - L) / 3 := by linarith
  obtain ⟨Tlim, hTlim⟩ := hlim ((a - L) / 3) hε
  obtain ⟨Tlower, hTlower⟩ := hlower ((a - L) / 3) hε
  let t : ℝ := max Tlim Tlower
  have htlim : Tlim ≤ t := by simp [t]
  have htlower : Tlower ≤ t := by simp [t]
  have habs : |f t - L| < (a - L) / 3 := hTlim t htlim
  have hu : f t - L < (a - L) / 3 := (abs_lt.mp habs).2
  have hl : a - (a - L) / 3 ≤ f t := hTlower t htlower
  linarith

theorem no_common_limit_of_separated_asymptotic_bounds
    {f g : ℝ → ℝ} {a b : ℝ}
    (hab : a < b)
    (hupper : EventuallyUpperBoundAtTop f a)
    (hlower : EventuallyLowerBoundAtTop g b) :
    ¬ ∃ L : ℝ, TendsToAtTop f L ∧ TendsToAtTop g L := by
  rintro ⟨L, hf, hg⟩
  have hLa : L ≤ a :=
    limit_le_of_tendsToAtTop_of_eventuallyUpperBound hf hupper
  have hbL : b ≤ L :=
    le_limit_of_tendsToAtTop_of_eventuallyLowerBound hg hlower
  linarith

/-- Ordering the level set functions reverses their positive time rates. -/
theorem eventuallyLowerBoundAtTop_rate_mono
    {G₁ G₂ : LevelSet} {x : Point} {c : ℝ}
    (horder : ∀ t, G₁ x t ≤ G₂ x t)
    (hlower : EventuallyLowerBoundAtTop (rateAt G₂ x) c) :
    EventuallyLowerBoundAtTop (rateAt G₁ x) c := by
  intro ε hε
  obtain ⟨T, hT⟩ := hlower ε hε
  refine ⟨max T 1, ?_⟩
  intro t ht
  have htT : T ≤ t := le_trans (le_max_left T 1) ht
  have ht1 : (1 : ℝ) ≤ t := le_trans (le_max_right T 1) ht
  have htpos : 0 < t := lt_of_lt_of_le zero_lt_one ht1
  have hbase : c - ε ≤ rateAt G₂ x t := hT t htT
  have hneg : -G₂ x t ≤ -G₁ x t := neg_le_neg (horder t)
  have hinv : 0 ≤ t⁻¹ := inv_nonneg.mpr (le_of_lt htpos)
  have hrate : rateAt G₂ x t ≤ rateAt G₁ x t := by
    simp only [rateAt, div_eq_mul_inv]
    exact mul_le_mul_of_nonneg_right hneg hinv
  exact hbase.trans hrate

/-- A positive uniform lower bound on `G` gives asymptotic rate at most zero. -/
theorem eventuallyUpperBoundAtTop_rate_of_positive_lower
    {G : LevelSet} {x : Point} {b : ℝ}
    (hb : 0 < b)
    (hG : ∀ t, 0 ≤ t → b ≤ G x t) :
    EventuallyUpperBoundAtTop (rateAt G x) 0 := by
  intro ε hε
  refine ⟨1, ?_⟩
  intro t ht
  have htpos : 0 < t := lt_of_lt_of_le zero_lt_one ht
  have hGt : b ≤ G x t := hG t (le_of_lt htpos)
  have hGpos : 0 < G x t := hb.trans_le hGt
  have hneg : -G x t ≤ 0 := neg_nonpos.mpr (le_of_lt hGpos)
  have hinv : 0 ≤ t⁻¹ := inv_nonneg.mpr (le_of_lt htpos)
  have hrate : rateAt G x t ≤ 0 := by
    simp only [rateAt, div_eq_mul_inv]
    exact mul_nonpos_of_nonpos_of_nonneg hneg hinv
  linarith

theorem no_phaseIndependentRate_of_incompatible_bounds
    {G : LevelSet} {c : ℝ}
    (hc : 0 < c)
    (hslow : EventuallyUpperBoundAtTop (rateAt G xSlow) 0)
    (hfast : EventuallyLowerBoundAtTop (rateAt G xFast) c) :
    ¬ HasPhaseIndependentRate G := by
  rintro ⟨L, hL⟩
  exact
    no_common_limit_of_separated_asymptotic_bounds hc hslow hfast
      ⟨L, hL xSlow, hL xFast⟩

/--
Conditional form of a sufficient `p = e₁` subregime of the manuscript's
main theorem. The currently unformalized inputs occur only as the explicit
hypotheses `hComparison`, `hXinYu`, `hControl`, and `hTrap`.
-/
theorem phaseDependentBurning_of_analytic_inputs
    {Control : Type} [Nonempty Control]
    (IsViscositySolution : ViscositySolutionPredicate)
    (Hunc Hcut Hhat : ℝ → ℝ → Hamiltonian)
    (Gunc Gcut Ghat : ℝ → ℝ → LevelSet)
    (trajectory : ℝ → ℝ → Point → Control → ℝ → Point)
    (hUncSolution : ∀ d A, IsViscositySolution (Hunc d A) planarE1 (Gunc d A))
    (hCutSolution : ∀ d A, IsViscositySolution (Hcut d A) planarE1 (Gcut d A))
    (hHatSolution : ∀ d A, IsViscositySolution (Hhat d A) planarE1 (Ghat d A))
    (hUncCut : ∀ d A x q, Hunc d A x q ≤ Hcut d A x q)
    (hCutHat : ∀ d A x q, Hcut d A x q ≤ Hhat d A x q)
    (hComparison : ViscosityComparisonAssumption IsViscositySolution)
    (hXinYu : XinYuFastPointAssumption Gunc)
    (hControl : ControlLowerEnvelopeAssumption trajectory Ghat)
    (hTrap : TrappingAssumption trajectory) :
    ∃ C : ℝ, 0 < C ∧
      ∀ d A : ℝ,
        0 < d →
        d < (1 : ℝ) / 60 →
        1 + 6 * d ^ 2 ≤ A * d →
        A * d ≤ 1 + d / 10 →
        (∀ t : ℝ, 0 ≤ t → Real.pi - 3 * d ≤ Gcut d A xSlow t) ∧
        0 < C * A / Real.log A ∧
        EventuallyLowerBoundAtTop
          (rateAt (Gcut d A) xFast)
          (C * A / Real.log A) ∧
        ¬ HasPhaseIndependentRate (Gcut d A) := by
  rcases hXinYu with ⟨C, hC, hXinYu⟩
  refine ⟨C, hC, ?_⟩
  intro d A hd0 hd60 hwindowLower hwindowUpper

  have hdQuarter : d < (1 : ℝ) / 4 := by
    linarith
  have hA4 : 4 ≤ A := by
    by_contra h
    have hAlt : A < 4 := lt_of_not_ge h
    have hAdlt : A * d < 4 * d := mul_lt_mul_of_pos_right hAlt hd0
    have h4dlt : 4 * d < 1 := by linarith
    have hd2pos : 0 < d ^ 2 := by positivity
    have hAdgt : 1 < A * d := by nlinarith
    linarith

  rcases hXinYu d A hd0 hdQuarter hA4 hwindowUpper with
    ⟨hFastConstantPositive, hUncFast⟩

  have hCutLeUnc : ∀ x t, Gcut d A x t ≤ Gunc d A x t :=
    hComparison (hUncCut d A) (hUncSolution d A) (hCutSolution d A)

  have hHatLeCut : ∀ x t, Ghat d A x t ≤ Gcut d A x t :=
    hComparison (hCutHat d A) (hCutSolution d A) (hHatSolution d A)

  have hHatSlow :
      ∀ t : ℝ, 0 ≤ t → Real.pi - 3 * d ≤ Ghat d A xSlow t := by
    intro t ht
    exact hControl d A xSlow t (Real.pi - 3 * d)
      (fun η => hTrap d A hd0 hd60 hwindowLower hwindowUpper η t ht)

  have hCutSlow :
      ∀ t : ℝ, 0 ≤ t → Real.pi - 3 * d ≤ Gcut d A xSlow t := by
    intro t ht
    exact (hHatSlow t ht).trans (hHatLeCut xSlow t)

  have hCutFast :
      EventuallyLowerBoundAtTop
        (rateAt (Gcut d A) xFast)
        (C * A / Real.log A) :=
    eventuallyLowerBoundAtTop_rate_mono
      (fun t => hCutLeUnc xFast t) hUncFast

  have hSlowConstantPositive : 0 < Real.pi - 3 * d := by
    nlinarith [Real.pi_gt_three]

  have hCutSlowRate :
      EventuallyUpperBoundAtTop (rateAt (Gcut d A) xSlow) 0 :=
    eventuallyUpperBoundAtTop_rate_of_positive_lower
      hSlowConstantPositive hCutSlow

  have hNoPhaseIndependentRate :
      ¬ HasPhaseIndependentRate (Gcut d A) :=
    no_phaseIndependentRate_of_incompatible_bounds
      hFastConstantPositive hCutSlowRate hCutFast

  exact ⟨hCutSlow, hFastConstantPositive, hCutFast, hNoPhaseIndependentRate⟩

end PhysicalCutoff
\end{lstlisting}

\section{A Conditional Lean~4 Formalization of the
Quantitative No Barrier Theorem}
\label{app:lean-no-barrier-formalization}

This appendix records a Lean~4 encoding of the logical assembly of
\Cref{thm:no barrier}. The local drift estimate, the contraction estimate
for the distance from a closed set, preservation of volume, continuity from
above for contracting tubular neighbourhoods, and the strict growth of
tubular volume are treated as explicit hypotheses. Conditional on those
inputs, the Lean development constructs a positive admissible value of
$\varepsilon_0$, proves
  $R(D)\le 2d\varepsilon$, 
and derives the exact barrier conclusion $D=\T_L^n$ when
$\varepsilon=0$.

No \texttt{sorry}, \texttt{admit}, proof hole, or globally declared axiom is
used in this conditional assembly. Every analytic or geometric input that is
not formalized in the file appears as an ordinary parameter of the final
theorem.

\subsection{Abstract Data and Explicit Hypotheses}

The Lean theorem
\texttt{quantitativeNoBarrier\_of\_analytic\_inputs} is stated for abstract
types of points and directions. It receives the Hamiltonian $H$, the vector
field $V$, its forward flow, the distance to a set, the covering radius, the
volume functional, and predicates encoding closedness, $C^2$ regularity,
incompressibility, the physical Hamiltonian relation, and unit proximal
outward normals. These data are intended to represent their corresponding
objects on the flat torus $\T_L^n$.

The theorem also receives the structural assumptions
\texttt{hC2}, \texttt{hDiv}, and \texttt{hPhysical}, asserting respectively
that $V$ is $C^2$, that $V$ is divergence free, and that $H$ is the physical
strain Hamiltonian generated by $V$ and $d$. The remaining inputs are the
following.

\begin{enumerate}[label=\textbf{(A\arabic*)}]
\item \texttt{hLocalDrift} has type
\texttt{LocalDriftAssumption}. It states that, under the dimension,
periodicity, regularity, incompressibility, and physical Hamiltonian
hypotheses, the barrier condition implies the local drift certificate used in
the proof. Its intended content is the estimate
\[
  V(y)\cdot\nu
  \le
  \varepsilon-
  \frac{\operatorname{dist}(y,D)}{2d}
\]
for points whose distance from $D$ lies in $(0,r_0)$.

\item \texttt{hContraction} has type
\texttt{DistanceContractionAssumption}. It packages the upper derivative
estimate for the distance and the ensuing scalar comparison. More precisely,
if $2d\varepsilon<r<r_0$, then for every $t\ge0$,
\[
  \Phi_t(D_r)
  \subset
  D_{r_t},
  \qquad
  r_t
  =
  2d\varepsilon+
  (r-2d\varepsilon)e^{-t/(2d)}.
\]

\item \texttt{hVolume} has type
\texttt{FlowVolumeAssumption}. Its first component states that divergence
free flows preserve volume, while its second component records monotonicity
of volume under set inclusion.

\item \texttt{hTubeLimit} has type
\texttt{TubeLimitAssumption}. It is the continuity from above step for the
contracting family $D_{r_t}$ and yields
  $|D_r|\le |D_{2d\varepsilon}|$ 
from the family of inequalities
$|D_r|\le |D_{r_t}|$.

\item \texttt{hTubeGeometry} has type
\texttt{TubularGeometryAssumption}. Its first component states that
\[
  0\le a<b<R(D)
  \quad\Longrightarrow\quad
  |D_a|<|D_b|,
\]
while its second component states that a nonempty closed set with covering
radius at most zero is the whole torus.
\end{enumerate}

In the final theorem declaration, these inputs appear explicitly as follows,

\begin{lstlisting}[style=lean4style,numbers=none]
(hC2 : IsC2 V)
(hDiv : IsDivergenceFree V)
(hPhysical : IsPhysicalHamiltonian V d H)
(hLocalDrift :
  LocalDriftAssumption n L d r₀ V H IsC2 IsDivergenceFree
    IsPhysicalHamiltonian IsClosedSet
    IsUnitProximalOutwardNormal HasLocalDrift)
(hContraction :
  DistanceContractionAssumption d r₀ flow distToSet
    HasLocalDrift)
(hVolume :
  FlowVolumeAssumption V IsDivergenceFree flow volume)
(hTubeLimit : TubeLimitAssumption d distToSet volume)
(hTubeGeometry :
  TubularGeometryAssumption IsClosedSet distToSet R volume)
\end{lstlisting}

No assumption is introduced globally. All the displayed inputs are ordinary
parameters of
\texttt{quantitativeNoBarrier\_of\_analytic\_inputs}.

\subsection{Logical Assembly}

The Lean proof chooses
  $\varepsilon_0
  =
  \frac{r_0}{4d}$, 
which is positive because $r_0>0$ and $d>0$. If
$0\le\varepsilon\le\varepsilon_0$, then
$2d\varepsilon<r_0$. Suppose, for contradiction, that
$R(D)>2d\varepsilon$, and choose
\[
  r
  =
  \frac{2d\varepsilon+
  \min\{R(D),r_0\}}{2}.
\]
Then, 
  $2d\varepsilon<r<\min\{R(D),r_0\}$.
The local drift and contraction assumptions give
$\Phi_t(D_r)\subset D_{r_t}$. Volume preservation and monotonicity therefore
give
\[
  |D_r|
  =
  |\Phi_t(D_r)|
  \le
  |D_{r_t}|
\]
for every $t\ge0$. The tube limit assumption yields
  $|D_r|\le |D_{2d\varepsilon}|$.
On the other hand, strict tubular growth gives
  $|D_{2d\varepsilon}|<|D_r|$,
a contradiction. Hence $R(D)\le2d\varepsilon$.

For an exact barrier, the quantitative estimate at $\varepsilon=0$ gives
$R(D)\le0$. The zero radius component of
\texttt{hTubeGeometry} then gives $D=\T_L^n$.

\subsection{Lean Source}

The complete Lean source is reproduced in
\Cref{lst:physical-cutoff-no-barrier-lean}. It defines all abstract data and
assumption predicates, proves the contradiction between contraction and
volume preservation, constructs $\varepsilon_0$, and derives both conclusions
of \Cref{thm:no barrier}.

\begin{lstlisting}[
  style=lean4style,
  caption={Conditional Lean~4 assembly of the quantitative no barrier
  theorem.},
  label={lst:physical-cutoff-no-barrier-lean}
]
import Mathlib

set_option autoImplicit false

namespace PhysicalCutoff.NoBarrier

universe u v

/-!
# Conditional assembly of outward-barrier-certificate rigidity

This file formalizes the logical assembly of the manuscript's rigidity
theorem for outward barrier certificates. The geometric and analytic inputs that are not formalized
here are explicit theorem parameters, not global declarations.

Conditional on those inputs, the development constructs a positive admissible
error threshold, proves the estimate `R D ≤ 2 * d * ε`, and derives the exact
barrier consequence `D = Set.univ` when `ε = 0`.
-/

abbrev Hamiltonian (Point : Type u) (Direction : Type v) :=
  Point → Direction → ℝ

abbrev VectorField (Point : Type u) (Direction : Type v) :=
  Point → Direction

abbrev Flow (Point : Type u) :=
  ℝ → Point → Point

abbrev DistanceToSet (Point : Type u) :=
  Point → Set Point → ℝ

abbrev CoveringRadius (Point : Type u) :=
  Set Point → ℝ

abbrev Volume (Point : Type u) :=
  Set Point → ℝ

abbrev ClosedSetPredicate (Point : Type u) :=
  Set Point → Prop

abbrev UnitProximalOutwardNormalPredicate
    (Point : Type u) (Direction : Type v) :=
  Set Point → Point → Direction → Prop

abbrev VectorFieldPredicate (Point : Type u) (Direction : Type v) :=
  VectorField Point Direction → Prop

abbrev PhysicalHamiltonianPredicate
    (Point : Type u) (Direction : Type v) :=
  VectorField Point Direction → ℝ → Hamiltonian Point Direction → Prop

abbrev LocalDriftPredicate (Point : Type u) :=
  Set Point → ℝ → Prop

/-- The Hamiltonian is at most `ε` at every unit proximal outward normal of `D`. -/
def BarrierCondition
    {Point : Type u} {Direction : Type v}
    (H : Hamiltonian Point Direction)
    (IsUnitProximalOutwardNormal :
      UnitProximalOutwardNormalPredicate Point Direction)
    (D : Set Point) (ε : ℝ) : Prop :=
  ∀ x, x ∈ D →
    ∀ ν, IsUnitProximalOutwardNormal D x ν →
      H x ν ≤ ε

/-- The closed `r`-neighbourhood of `D`, expressed through an abstract distance-to-set map. -/
def ClosedTube
    {Point : Type u}
    (distToSet : DistanceToSet Point)
    (D : Set Point) (r : ℝ) : Set Point :=
  {x | distToSet x D ≤ r}

/--
The local analytic input. Under the dimension, periodicity-scale, positivity,
regularity, incompressibility, and physical-Hamiltonian hypotheses, a barrier
condition produces the local drift certificate used by the distance argument.
-/
def LocalDriftAssumption
    {Point : Type u} {Direction : Type v}
    (n : ℕ) (L d r₀ : ℝ)
    (V : VectorField Point Direction)
    (H : Hamiltonian Point Direction)
    (IsC2 : VectorFieldPredicate Point Direction)
    (IsDivergenceFree : VectorFieldPredicate Point Direction)
    (IsPhysicalHamiltonian : PhysicalHamiltonianPredicate Point Direction)
    (IsClosedSet : ClosedSetPredicate Point)
    (IsUnitProximalOutwardNormal :
      UnitProximalOutwardNormalPredicate Point Direction)
    (HasLocalDrift : LocalDriftPredicate Point) : Prop :=
  2 ≤ n →
  0 < L →
  0 < d →
  0 < r₀ →
  IsC2 V →
  IsDivergenceFree V →
  IsPhysicalHamiltonian V d H →
  ∀ {D : Set Point} {ε : ℝ},
    D.Nonempty →
    IsClosedSet D →
    BarrierCondition H IsUnitProximalOutwardNormal D ε →
    HasLocalDrift D ε

/--
The distance-evolution input. A local drift certificate yields the exponentially
contracting tube inclusion from the manuscript.
-/
def DistanceContractionAssumption
    {Point : Type u}
    (d r₀ : ℝ)
    (flow : Flow Point)
    (distToSet : DistanceToSet Point)
    (HasLocalDrift : LocalDriftPredicate Point) : Prop :=
  ∀ {D : Set Point} {ε r : ℝ},
    HasLocalDrift D ε →
    0 ≤ ε →
    2 * d * ε < r →
    r < r₀ →
    ∀ t : ℝ, 0 ≤ t →
      flow t '' ClosedTube distToSet D r ⊆
        ClosedTube distToSet D
          (2 * d * ε +
            (r - 2 * d * ε) * Real.exp (-t / (2 * d)))

/--
The measure-theoretic input: the flow preserves volume, and volume is monotone
under set inclusion.
-/
def FlowVolumeAssumption
    {Point : Type u} {Direction : Type v}
    (V : VectorField Point Direction)
    (IsDivergenceFree : VectorFieldPredicate Point Direction)
    (flow : Flow Point)
    (volume : Volume Point) : Prop :=
  (IsDivergenceFree V →
    ∀ (t : ℝ) (S : Set Point),
      volume (flow t '' S) = volume S) ∧
  (∀ {S T : Set Point},
    S ⊆ T → volume S ≤ volume T)

/--
Continuity from above for the contracting family of closed tubes appearing in
the proof.
-/
def TubeLimitAssumption
    {Point : Type u}
    (d : ℝ)
    (distToSet : DistanceToSet Point)
    (volume : Volume Point) : Prop :=
  ∀ {D : Set Point} {a r : ℝ},
    0 ≤ a →
    a < r →
    (∀ t : ℝ, 0 ≤ t →
      volume (ClosedTube distToSet D r) ≤
        volume
          (ClosedTube distToSet D
            (a + (r - a) * Real.exp (-t / (2 * d))))) →
    volume (ClosedTube distToSet D r) ≤
      volume (ClosedTube distToSet D a)

/--
The geometric input for tubular neighbourhoods: their volume grows strictly
below the covering radius, and covering radius zero forces a nonempty closed
set to be the whole state space.
-/
def TubularGeometryAssumption
    {Point : Type u}
    (IsClosedSet : ClosedSetPredicate Point)
    (distToSet : DistanceToSet Point)
    (R : CoveringRadius Point)
    (volume : Volume Point) : Prop :=
  (∀ {D : Set Point} {a b : ℝ},
    D.Nonempty →
    IsClosedSet D →
    0 ≤ a →
    a < b →
    b < R D →
    volume (ClosedTube distToSet D a) <
      volume (ClosedTube distToSet D b)) ∧
  (∀ {D : Set Point},
    D.Nonempty →
    IsClosedSet D →
    R D ≤ 0 →
    D = Set.univ)

/--
The contradiction between tube contraction, volume preservation, and strict
tubular growth gives the quantitative covering-radius estimate.
-/
theorem coveringRadius_le_of_contraction_and_volume
    {Point : Type u}
    (d r₀ ε : ℝ)
    (flow : Flow Point)
    (distToSet : DistanceToSet Point)
    (R : CoveringRadius Point)
    (volume : Volume Point)
    (IsClosedSet : ClosedSetPredicate Point)
    (HasLocalDrift : LocalDriftPredicate Point)
    (D : Set Point)
    (hd : 0 < d)
    (hε : 0 ≤ ε)
    (hscale : 2 * d * ε < r₀)
    (hDNonempty : D.Nonempty)
    (hDClosed : IsClosedSet D)
    (hLocalD : HasLocalDrift D ε)
    (hContraction :
      DistanceContractionAssumption d r₀ flow distToSet HasLocalDrift)
    (hVolumePreserving :
      ∀ (t : ℝ) (S : Set Point),
        volume (flow t '' S) = volume S)
    (hVolumeMono :
      ∀ {S T : Set Point}, S ⊆ T → volume S ≤ volume T)
    (hTubeLimit : TubeLimitAssumption d distToSet volume)
    (hStrictGrowth :
      ∀ {D : Set Point} {a b : ℝ},
        D.Nonempty →
        IsClosedSet D →
        0 ≤ a →
        a < b →
        b < R D →
        volume (ClosedTube distToSet D a) <
          volume (ClosedTube distToSet D b)) :
    R D ≤ 2 * d * ε := by
  by_contra hRadius
  have hRadiusLarge : 2 * d * ε < R D :=
    lt_of_not_ge hRadius
  have hBelowMin : 2 * d * ε < min (R D) r₀ :=
    (lt_min_iff).2 ⟨hRadiusLarge, hscale⟩

  let r : ℝ := (2 * d * ε + min (R D) r₀) / 2

  have hLower : 2 * d * ε < r := by
    dsimp [r]
    linarith

  have hUpperMin : r < min (R D) r₀ := by
    dsimp [r]
    linarith

  have hUpperRadius : r < R D :=
    lt_of_lt_of_le hUpperMin (min_le_left (R D) r₀)

  have hUpperScale : r < r₀ :=
    lt_of_lt_of_le hUpperMin (min_le_right (R D) r₀)

  have hBaseNonnegative : 0 ≤ 2 * d * ε := by
    positivity

  have hVolumeFamily :
      ∀ t : ℝ, 0 ≤ t →
        volume (ClosedTube distToSet D r) ≤
          volume
            (ClosedTube distToSet D
              (2 * d * ε +
                (r - 2 * d * ε) * Real.exp (-t / (2 * d)))) := by
    intro t ht
    have hInclusion :=
      hContraction hLocalD hε hLower hUpperScale t ht
    calc
      volume (ClosedTube distToSet D r) =
          volume (flow t '' ClosedTube distToSet D r) := by
            symm
            exact hVolumePreserving t (ClosedTube distToSet D r)
      _ ≤ volume
          (ClosedTube distToSet D
            (2 * d * ε +
              (r - 2 * d * ε) * Real.exp (-t / (2 * d)))) :=
        hVolumeMono hInclusion

  have hContractedVolume :
      volume (ClosedTube distToSet D r) ≤
        volume (ClosedTube distToSet D (2 * d * ε)) :=
    hTubeLimit hBaseNonnegative hLower hVolumeFamily

  have hStrictVolume :
      volume (ClosedTube distToSet D (2 * d * ε)) <
        volume (ClosedTube distToSet D r) :=
    hStrictGrowth hDNonempty hDClosed hBaseNonnegative hLower hUpperRadius

  exact (not_lt_of_ge hContractedVolume) hStrictVolume

/--
Conditional form of the manuscript's rigidity theorem for outward barrier
certificates. Every
unformalized geometric or analytic input occurs as an explicit hypothesis.
-/
theorem quantitativeNoBarrier_of_analytic_inputs
    {Point : Type u} {Direction : Type v}
    (n : ℕ) (L d r₀ : ℝ)
    (V : VectorField Point Direction)
    (H : Hamiltonian Point Direction)
    (flow : Flow Point)
    (distToSet : DistanceToSet Point)
    (R : CoveringRadius Point)
    (volume : Volume Point)
    (IsC2 : VectorFieldPredicate Point Direction)
    (IsDivergenceFree : VectorFieldPredicate Point Direction)
    (IsPhysicalHamiltonian : PhysicalHamiltonianPredicate Point Direction)
    (IsClosedSet : ClosedSetPredicate Point)
    (IsUnitProximalOutwardNormal :
      UnitProximalOutwardNormalPredicate Point Direction)
    (HasLocalDrift : LocalDriftPredicate Point)
    (hn : 2 ≤ n)
    (hL : 0 < L)
    (hd : 0 < d)
    (hr₀ : 0 < r₀)
    (hC2 : IsC2 V)
    (hDiv : IsDivergenceFree V)
    (hPhysical : IsPhysicalHamiltonian V d H)
    (hLocalDrift :
      LocalDriftAssumption n L d r₀ V H IsC2 IsDivergenceFree
        IsPhysicalHamiltonian IsClosedSet IsUnitProximalOutwardNormal
        HasLocalDrift)
    (hContraction :
      DistanceContractionAssumption d r₀ flow distToSet HasLocalDrift)
    (hVolume :
      FlowVolumeAssumption V IsDivergenceFree flow volume)
    (hTubeLimit : TubeLimitAssumption d distToSet volume)
    (hTubeGeometry :
      TubularGeometryAssumption IsClosedSet distToSet R volume) :
    ∃ ε₀ : ℝ, 0 < ε₀ ∧
      (∀ (ε : ℝ),
        0 ≤ ε →
        ε ≤ ε₀ →
        ∀ D : Set Point,
          D.Nonempty →
          IsClosedSet D →
          BarrierCondition H IsUnitProximalOutwardNormal D ε →
          R D ≤ 2 * d * ε) ∧
      (∀ D : Set Point,
        D.Nonempty →
        IsClosedSet D →
        BarrierCondition H IsUnitProximalOutwardNormal D 0 →
        D = Set.univ) := by
  rcases hVolume with ⟨hVolumePreservingOfDiv, hVolumeMono⟩
  rcases hTubeGeometry with ⟨hStrictGrowth, hZeroRadius⟩

  have hVolumePreserving :
      ∀ (t : ℝ) (S : Set Point),
        volume (flow t '' S) = volume S :=
    hVolumePreservingOfDiv hDiv

  let ε₀ : ℝ := r₀ / (4 * d)

  have hFourD : 0 < 4 * d := by
    positivity

  have hε₀Positive : 0 < ε₀ := by
    dsimp [ε₀]
    exact div_pos hr₀ hFourD

  have hQuantitative :
      ∀ (ε : ℝ),
        0 ≤ ε →
        ε ≤ ε₀ →
        ∀ D : Set Point,
          D.Nonempty →
          IsClosedSet D →
          BarrierCondition H IsUnitProximalOutwardNormal D ε →
          R D ≤ 2 * d * ε := by
    intro ε hε hεSmall D hDNonempty hDClosed hBarrier

    have hεSmall' : ε ≤ r₀ / (4 * d) := by
      simpa [ε₀] using hεSmall

    have hScaledBound : ε * (4 * d) ≤ r₀ :=
      (le_div_iff₀ hFourD).mp hεSmall'

    have hScale : 2 * d * ε < r₀ := by
      nlinarith

    have hLocalD : HasLocalDrift D ε :=
      hLocalDrift hn hL hd hr₀ hC2 hDiv hPhysical
        hDNonempty hDClosed hBarrier

    exact coveringRadius_le_of_contraction_and_volume
      d r₀ ε flow distToSet R volume IsClosedSet HasLocalDrift D
      hd hε hScale hDNonempty hDClosed hLocalD hContraction
      hVolumePreserving hVolumeMono hTubeLimit hStrictGrowth

  have hExact :
      ∀ D : Set Point,
        D.Nonempty →
        IsClosedSet D →
        BarrierCondition H IsUnitProximalOutwardNormal D 0 →
        D = Set.univ := by
    intro D hDNonempty hDClosed hBarrier
    have hRadius : R D ≤ 2 * d * 0 :=
      hQuantitative 0 le_rfl (le_of_lt hε₀Positive)
        D hDNonempty hDClosed hBarrier
    have hRadiusZero : R D ≤ 0 := by
      simpa using hRadius
    exact hZeroRadius hDNonempty hDClosed hRadiusZero

  exact ⟨ε₀, hε₀Positive, hQuantitative, hExact⟩

end PhysicalCutoff.NoBarrier
\end{lstlisting}


\begin{thebibliography}{99}

\bibitem{AugustinEtAl2014}
T. Augustin, F. P. A. Coolen, G. de Cooman, and M. C. M. Troffaes, eds.,
\emph{Introduction to Imprecise Probabilities}, Wiley, Chichester, 2014.

\bibitem{BardiCapuzzoDolcetta1997}
M. Bardi and I. Capuzzo-Dolcetta,
\emph{Optimal Control and Viscosity Solutions of Hamilton-Jacobi-Bellman Equations},
Birkh\"auser, Boston, 1997.

\bibitem{BardiTerrone2013}
M. Bardi and G. Terrone,
On the homogenization of some non-coercive Hamilton-Jacobi-Isaacs equations,
\emph{Commun. Pure Appl. Anal.} \textbf{12} (2013), 207--236.

\bibitem{Barles2007}
G. Barles,
Some homogenization results for non-coercive Hamilton-Jacobi equations,
\emph{Calc. Var. Partial Differential Equations} \textbf{30} (2007), 449--466.

\bibitem{BergerRiosInsuaRuggeri2000}
J. O. Berger, D. R\'ios Insua, and F. Ruggeri,
Bayesian robustness,
in \emph{Robust Bayesian Analysis},
D. R\'ios Insua and F. Ruggeri, eds.,
Lecture Notes in Statistics, vol.~152,
Springer, New York, 2000, pp.~1--32.


\bibitem{CaffarelliMonneau2014}
L. A. Caffarelli and R. Monneau,
Counter-example in three dimension and homogenization of geometric motions in two dimension,
\emph{Arch. Ration. Mech. Anal.} \textbf{212} (2014), 503--574.

\bibitem{CaprioChen2026}
M. Caprio and M. Chen,
Imprecise Markov semigroups and their ergodicity,
\emph{ESAIM Probab. Stat.}, 2026, doi:10.1051/ps/2026007.

\bibitem{CaprioEtAl2024}
M. Caprio, S. Dutta, K. J. Jang, V. Lin, R. Ivanov, O. Sokolsky, and I. Lee,
Credal Bayesian deep learning,
\emph{Trans. Mach. Learn. Res.}, 2024.

\bibitem{CaprioSultanaEliaCuzzolin2024}
M. Caprio, M. Sultana, E. Elia, and F. Cuzzolin,
Credal learning theory,
arXiv:2402.00957, 2024.

\bibitem{Cardaliaguet2010}
P. Cardaliaguet,
Ergodicity of Hamilton-Jacobi equations with a noncoercive nonconvex
Hamiltonian in $\mathbb R^2/\mathbb Z^2$,
\emph{Ann. Inst. H. Poincar\'e Anal. Non Lin\'eaire}
\textbf{27} (2010), 837--856.

\bibitem{CardaliaguetLionsSouganidis2009}
P. Cardaliaguet, P.-L. Lions, and P. E. Souganidis,
A discussion about the homogenization of moving interfaces,
\emph{J. Math. Pures Appl.} \textbf{91} (2009), 339--363.

\bibitem{CardaliaguetNolenSouganidis2011}
P. Cardaliaguet, J. Nolen, and P. E. Souganidis,
Homogenization and enhancement for the $G$-equation in periodic media,
\emph{Arch. Ration. Mech. Anal.} \textbf{199} (2011), 527--561.

\bibitem{ChenBerrettDamoulasCaprio2026}
M. Chen, T. B. Berrett, T. Damoulas, and M. Caprio,
Bulk-calibrated credal ambiguity sets: fast, tractable decision making
under out-of-sample contamination,
in \emph{Proceedings of the 43rd International Conference on Machine
Learning},
\emph{Proc. Mach. Learn. Res.} \textbf{306} (2026).

\bibitem{CrandallIshiiLions1992}
M. G. Crandall, H. Ishii, and P.-L. Lions,
User's guide to viscosity solutions of second order partial
differential equations,
\emph{Bull. Amer. Math. Soc. (N.S.)}
\textbf{27} (1992), 1--67.

\bibitem{DaviniSaonaZiliotto2026}
A. Davini, R. Saona, and B. Ziliotto,
Stochastic homogenization of Hamilton-Jacobi equations: a differential game approach,
\emph{Ann. Inst. H. Poincar\'e C Anal. Non Lin\'eaire} \textbf{43} (2026), 883--925.

\bibitem{DeCoomanHermansQuaeghebeur2009}
G. de Cooman, F. Hermans, and E. Quaeghebeur,
Imprecise Markov chains and their limit behavior,
\emph{Probab. Engrg. Inform. Sci.} \textbf{23} (2009), 597--635.

\bibitem{DuchiGlynnNamkoong2021}
J. C. Duchi, P. W. Glynn, and H. Namkoong,
Statistics of robust optimization: a generalized empirical likelihood approach,
\emph{Math. Oper. Res.} \textbf{46} (2021), 946--969.


\bibitem{EvansSouganidis1984}
L. C. Evans and P. E. Souganidis,
Differential games and representation formulas for solutions of Hamilton-Jacobi--Isaacs equations,
\emph{Indiana Univ. Math. J.} \textbf{33} (1984), 773--797.

\bibitem{GaoLongXinYu2024}
H. Gao, Z. Long, J. Xin, and Y. Yu,
Existence of an effective burning velocity in a cellular flow for the curvature $G$-equation proved using a game analysis,
\emph{J. Geom. Anal.} \textbf{34} (2024), Article 81.

\bibitem{HermansDeCooman2012}
F. Hermans and G. de Cooman,
Characterisation of ergodic upper transition operators,
\emph{Internat. J. Approx. Reason.} \textbf{53} (2012), 573--583.

\bibitem{HuellermeierWaegeman2021}
E. H\"ullermeier and W. Waegeman,
Aleatoric and epistemic uncertainty in machine learning: an introduction to concepts and methods,
\emph{Machine Learning} \textbf{110} (2021), 457--506.

\bibitem{Isaacs1965}
R. Isaacs,
\emph{Differential Games}, Wiley, New York, 1965.

\bibitem{Iyengar2005}
G. N. Iyengar,
Robust dynamic programming,
\emph{Math. Oper. Res.} \textbf{30} (2005), 257--280.

\bibitem{Levi1980}
I. Levi,
\emph{The Enterprise of Knowledge: An Essay on Knowledge, Credal Probability, and Chance},
MIT Press, Cambridge, MA, 1980.

\bibitem{LiuXinYu2013}
Y.-Y. Liu, J. Xin, and Y. Yu,
A numerical study of turbulent flame speeds of curvature and strain $G$-equations in cellular flows,
\emph{Phys. D} \textbf{243} (2013), 20--31.

\bibitem{Markstein1964}
G. H. Markstein, ed.,
\emph{Nonsteady Flame Propagation}, AGARDograph 75,
Pergamon Press, Oxford, 1964.

\bibitem{MitakeMooneyTranXinYu2025}
H. Mitake, C. Mooney, H. V. Tran, J. Xin, and Y. Yu,
Bifurcation of homogenization and nonhomogenization of the curvature $G$-equation with shear flows,
\emph{Math. Ann.} \textbf{391} (2025), 3077--3111.

\bibitem{NilimElGhaoui2005}
A. Nilim and L. El Ghaoui,
Robust control of Markov decision processes with uncertain transition matrices,
\emph{Operations Research} \textbf{53} (2005), 780--798.

\bibitem{OsherSethian1988}
S. Osher and J. A. Sethian,
Fronts propagating with curvature-dependent speed: algorithms based on Hamilton-Jacobi formulations,
\emph{J. Comput. Phys.} \textbf{79} (1988), 12--49.

\bibitem{Peng2007}
S. Peng,
$G$-expectation, $G$-Brownian motion and related stochastic calculus of
It\^o type,
in \emph{Stochastic Analysis and Applications},
Abel Symp., vol.~2,
Springer, Berlin, 2007, pp.~541--567.

\bibitem{Peng2008}
S. Peng,
A new central limit theorem under sublinear expectations,
arXiv:0803.2656, 2008.

\bibitem{Peters2000}
N. Peters,
\emph{Turbulent Combustion}, Cambridge University Press, Cambridge, 2000.

\bibitem{Puterman1994}
M. L. Puterman,
\emph{Markov Decision Processes: Discrete Stochastic Dynamic Programming},
Wiley, New York, 1994.

\bibitem{TJoensDeBock2021}
N. T'Joens and J. De Bock,
Average behaviour in discrete time imprecise Markov chains: a study of weak ergodicity,
\emph{Internat. J. Approx. Reason.} \textbf{132} (2021), 181--205.

\bibitem{TroffaesDeCooman2014}
M. C. M. Troffaes and G. de Cooman,
\emph{Lower Previsions}, Wiley, Chichester, 2014.

\bibitem{Walley1991}
P. Walley,
\emph{Statistical Reasoning with Imprecise Probabilities}, Chapman and Hall, London, 1991.

\bibitem{WiesemannKuhnRustem2013}
W. Wiesemann, D. Kuhn, and B. Rustem,
Robust Markov decision processes,
\emph{Math. Oper. Res.} \textbf{38} (2013), 153--183.

\bibitem{Williams1985}
F. A. Williams,
Turbulent combustion,
in \emph{The Mathematics of Combustion}, J. D. Buckmaster, ed., SIAM, Philadelphia, 1985, pp. 97--131.

\bibitem{XinYu2010}
J. Xin and Y. Yu,
Periodic homogenization of the inviscid $G$-equation for incompressible flows,
\emph{Commun. Math. Sci.} \textbf{8} (2010), 1067--1078.

\bibitem{XinYu2014}
J. Xin and Y. Yu,
Front quenching in the $G$-equation model induced by straining of cellular flow,
\emph{Arch. Ration. Mech. Anal.} \textbf{214} (2014), 1--34.

\bibitem{XinYuRonney2024}
J. Xin, Y. Yu, and P. Ronney,
Lagrangian, game theoretic and PDE methods for averaging $G$-equations in turbulent combustion: existence and beyond,
\emph{Bull. Amer. Math. Soc.} \textbf{61} (2024), 470--514.

\end{thebibliography}
\end{document}